\documentclass[a4paper,10pt]{article}
\usepackage[english]{babel}
\usepackage[utf8]{inputenc} 
\usepackage{amsmath,amssymb,amsthm,esint}
\usepackage[a4paper]{geometry}
\usepackage[cmyk]{xcolor}
\usepackage{silence}
\usepackage{pict2e}
\usepackage[unicode=true,pdfusetitle,bookmarks=true,bookmarksnumbered=false,bookmarksopen=false,
breaklinks=false,pdfborder={0 0 0},pdfborderstyle={},backref=false,colorlinks=true]{hyperref}
\usepackage[capitalize]{cleveref}
\usepackage{comment}
\usepackage[shortlabels]{enumitem}
\usepackage{graphicx}
\usepackage{cancel}
\usepackage{tikz}
\usetikzlibrary{shapes.geometric}
\usetikzlibrary{arrows.meta,arrows}
\usetikzlibrary{patterns}
\usepackage{appendix}

\usepackage[final]{showlabels}
\usepackage[final]{showkeys}

\newtheorem{theorem}{Theorem}[section]
\newtheorem{lemma}[theorem]{Lemma}
\newtheorem{proposition}[theorem]{Proposition}

\newtheorem{definition}[theorem]{Definition}

\newtheorem{remark}[theorem]{Remark}
\newtheorem{corollary}[theorem]{Corollary}
\newcommand{\be}{\begin{equation}}
\newcommand{\ee}{\end{equation}}
\newcommand{\ba}{\begin{aligned}}
	\newcommand{\ea}{\end{aligned}}
\newcommand{\R}{\mathbb R}
\newcommand{\N}{\mathbb N}
\newcommand{\bC}{\mathbb C}
\newcommand{\Div}{\mathrm{div}\;}
\newcommand{\Om}{\Omega}
\newcommand{\p}{\partial}
\newcommand{\eps}{\varepsilon}
\newcommand{\nueff}{\nu_{\mathrm{eff}}}

\newcommand{\om}{\omega}
\newcommand{\sgn}{\mathrm{sgn}\;}

\newcommand{\dd}{\mathrm{d}}
\newcommand{\hS}{\widehat{S}}

\newcommand{\cG}{\widehat{G}}
\newcommand{\cF}{\mathcal F}

\newcommand{\uapp}{u^{\mathrm{app}}}
\newcommand{\divh}{\mathrm{div}_h\,}
\newcommand{\tu}{\widetilde{u}}

\numberwithin{equation}{section}

\title{A linear model of separation for western boundary currents with bathymetry and stratification}

\author{Anne-Laure Dalibard\footnote{
Sorbonne Université, Université Paris Cité,
CNRS, INRIA;
Laboratoire Jacques-Louis Lions, LJLL, EPC ANGE;
75005 Paris, France;
Email: anne-laure.dalibard@sorbonne-universite.fr}  \  and Corentin Gentil\footnote{DMA, École normale supérieure, Université PSL, CNRS, 75005 Paris, France}}

\begin{document}
\bibliographystyle{plain}

\maketitle

\begin{abstract}
    This paper is devoted to the asymptotic analysis of strongly rotating and stratified fluids, under a $\beta$-plane approximation, and within a three-dimensional spatial domain with strong topography.
Our purpose is to propose a linear idealized model, which is able to capture one of the key features of western boundary currents, in spite of its simplicity: the separation of the currents from the coast.
Our simplified framework allows us to perform  explicit computations, and to highlight the intricate links between rotation, stratification and bathymetry.
In fact, we are able to construct approximate solutions at any order for our system, and to justify their validity.
Each term in the asymptotic expansion is the sum of an interior part and of two boundary layer parts: a ``Munk'' type boundary layer, which is quasi-geostrophic, and an ``Ekman part'', which is not.
Even though the Munk part of the approximation bears some similarity with previously studied 2D models, the analysis of the Ekman part is completely new, and several of its properties differ strongly from the ones of classical Ekman layers.
Our theoretical analysis is supplemented with numerical illustrations, which exhibit the desired separation behavior.

\end{abstract}





\section{Introduction}

The purpose of this paper is to perform an asymptotic analysis of the system
\begin{equation}\label{NS-rho}
\begin{aligned}
\p_t u + \frac{1}{\varepsilon} (1 + \varepsilon \beta y)e_3 \wedge u + \frac{1}{\varepsilon}\begin{pmatrix}\nabla_h p\\ \delta^{-2} \p_3 p \end{pmatrix}  - \nu_h \Delta_h u - \nu_3 \p_3^2 u=&\ \frac{1}{\eps \delta^2}\begin{pmatrix}
    0\\ 0\\ -  \rho
\end{pmatrix}+\beta f\quad \text{in }\Om,  \\
\p_t \rho - \frac{1}{\varepsilon } u_3=&\ 0\quad \text{in }\Om,   \\
\Div u=&\ 0\quad \text{in }\Om,  \\
u\vert _{\p\Om}=&\ 0,  
\end{aligned} 
\end{equation}
in the regime $\eps,\nu_h,\nu_3,\delta\ll 1$ and $\beta\gg 1$. In fact, we will choose all parameters as powers of $\eps$, in a way that we will specify later (see assumption \eqref{hyp:parameters} below).
This system is a linearized version of the rotating Boussinesq model in a $\beta$-plane approximation, and can be seen as an idealized toy model for the behaviour of oceanic currents on large horizontal scales. 
The function $u$ is the 3D velocity, $p$ is the pressure, and $\rho$ is the variation of the density, that is, the buoyancy. Let us also mention that the parameter $\eps$ is both the Rossby number and the Brunt-Väisälä frequency, $\delta$ is the aspect ratio of the fluid domain, $\nu_h$ and $\nu_3$ are the rescaled horizontal and vertical eddy diffusivities.
The parameter $\beta$ stems from a Taylor expansion of the Coriolis factor.
We provide a short, formal derivation of \eqref{NS-rho} in Appendix \ref{appendix-derivation} in this context, together with a notion of weak solution.

One of the main novelties of our study, compared to previous works, lies in the geometry of the domain $\Om$.
Indeed, we are interested in the case where the bottom boundary of $\Om$ is \textit{not} a flat horizontal surface. 
Our goal is to understand the influence of the topography on the dynamics of oceanic currents, and more specifically its interaction with the stratification of the fluid.
Thus we consider the idealized case where 
 $\Om$ is an infinite half-plane over a tilted surface, namely $\Om=\{(x_1, x_2, x_3)\in \R^3,\ x_3> \tan \alpha x_1\}$, with $\alpha \in (-\pi/2, \pi/2)$.
This simple geometry will allow us to perform explicit computations and to construct an approximate solution of system \eqref{NS-rho} at an arbitrarily high order.
Such an idealized configuration is very similar to the one studied by Pedlosky in \cite{pedloskywave}, Lecture 19, \textit{Topographic Waves in a Stratified Fluid}, partly with the same ingredients : a rotating stratified fluid interacting with topography.

The source term $f$ models external forces, such as the wind forcing, acting on the system. 
It is assumed to be time periodic, with a given frequency $\om$.
Although the wind forcing rather acts at the surface of the fluid as a boundary condition, we describe here its influence as a volumic source term, since our domain does not have an upper boundary.
However, we believe that our methodology could be applied to more general geometries, and in particular to 3D oceanic basins which are bounded in the vertical direction. 

The main achievements of the paper are as follows.
First, as mentioned above, we are able to construct an approximate solution at an arbitrary order, provided the source term $f$ is sufficiently smooth. This solution can be computed analytically from the forcing $f$, and we can use the semi-explicit formula to plot it in Python. The numerical simulations give a behaviour very similar to realistic configurations, see Section \ref{sec:results} for illustrations. 
We precisely chose to work with a linear model in order to be able 
 to plot a semi-explicit solution, though we could have added non-linear terms as long as they can be treated perturbatively. However, let us just emphasize that dominant non-linear effects would oblige us to change drastically our strategy.

This approximate solution is quasi-geostrophic, and is obtained as an asymptotic expansion in powers of the parameter $\eps\beta\ll 1$. 
Each term in the expansion is itself the sum of three terms: one interior term, one ``Munk type'' boundary layer term, and one ``Ekman type'' boundary layer, whose width will be much smaller than the Munk boundary layer term.
Although the construction of Munk boundary layers is similar to previous works in the presence of flat topographies and vertical coastlines (i.e. when $\Om=(0,1)^3$, for instance), the construction of Ekman layers in the present  context (strong topography, importance of the stratification) is completely new.
Of course it is also crucial to understand the interplay between the different components of the approximate solution.
Second, we will prove the stability and the validity of this approximate solution in two different frameworks (time-periodic solutions and Cauchy problem).

\subsection{Motivation from physical oceanography} \label{sec:motivation}

Numerous simplified versions of the Navier--Stokes equations adapted to the ocean configuration exist in physical oceanography. 
Our long-term interest  is to describe the interaction between western boundary currents (such as the Gulf Stream) and topography, and more specifically how western boundary currents separate from the coast. 
Let us now explain some of the modelling considerations which led us to choose \eqref{NS-rho} to describe the phenomenon we are interested in.

First and foremost, the $\beta$-plane approximation (\textit{i.e.} a Taylor expansion at order 1 of the Coriolis force around a given latitude, rather than at order 0)
is necessary to see western boundary currents emerge (see \cite{pedlosky1996ocean}).
Second, we require our model to take into account topographies that vary at the same order of magnitude than the water depth. Indeed, where western boundary currents separate from the coast, the depth of the ocean floor can change from a few hundred metres to several kilometres over a small horizontal scale (a few tens of kilometres). Thus, the (often used) assumption of small variations in topography relative to water height (see for instance \cite{desjardins1999homogeneous,masmoudi2000ekman}) does not apply to the separation phenomenon we wish to describe here. 

This last point  led us to model a stratified ocean and describe a 3D stream function rather than a 2D one.
Let us give a bit more details. If the ocean is assumed to be homogeneous, i.e. with constant density, then $\rho=0$ in \eqref{NS-rho} (recall that $\rho$ represents the density variations).
As a consequence, looking formally at the main order terms in \eqref{NS-rho} as $\eps \to 0$ and taking the parameters $\eps$ and $\beta$ so that $\eps\beta\ll 1$, we obtain
\[
u_h^\bot + \nabla_h p =0,\qquad \p_3 p=0,
\]
where $u_h^\bot= (-u_2, u_1)$. It follows that the motion is geostrophic at main order, and no variation on the vertical occurs, therefore it is described by a 2D stream function: this is the Taylor--Proudman theorem.
Using the divergence free condition, we also find that $u_3=0$.
Now, assume that $\Om=\{x_3>\eta_B(x_h)\}$ for some smooth function $\eta_b (x_h)$ (we choose $\eta_B=x_1 \tan \alpha$ in the present paper, but we consider general topographies in this paragraph for the sake of discussion).
Enforcing the condition $u\cdot n=0$ on the bottom boundary, we find that $\nabla_h^\bot p(x_h)\cdot \nabla_h \eta_B(x_h)=0$.
When $\eta_B$ is not constant, a solution of this equation is given by $p= \Phi(\eta_B)$ for some smooth function $\Phi$, and in this case $u_h$ is colinear to $\nabla_h^\bot \eta_B$ everywhere.
In conclusion, we find that in the fast rotation limit, the
flow is forced to follow the isobaths. Thus a western boundary current running along the coast could not bifurcate towards the inner ocean, as it would cross the isobaths. 
Hence we work with a non-homogeneous model in order to describe properly the western boundary currents.

This constraint for rapidly rotating homogeneous fluids is well known from a physics perspective, see for example \cite{Pedlosky-GFD}, and was recently demonstrated mathematically in the case of a fluid above a topography with non-small variations in \cite{chemin2024ekman} (apart from the homogeneity of the fluid, the other differences in assumptions we make compared to \cite{chemin2024ekman} are the linearity of our model and the presence of the $\beta$ effect (not to be confused with the coefficient $\beta$ in \cite{chemin2024ekman})).

Numerical experiments dating back to the 1970s (\cite{Holland1973}) showed that by simultaneously taking into account topography and stratification in an idealised ocean basin configuration with a western boundary current, the results obtained were significantly more realistic than those obtained using only stratification or only topography. Subsequently, it was shown that the term describing the interaction between topography and stratification (called JEBAR, for ‘‘Joint Effect of Baroclinicity And Relief'') was key to understanding vorticity balances at the level of western boundary currents. In particular, in \cite{schoonover2016}, it was shown that the term associated with the $\beta$ effect (i.e., responsible for western boundary currents) is balanced by a term close to the JEBAR term. This of course does not guarantee that other terms do not play a role, but rather indicates that it is essential to take this effect into account.

However, while such studies have indeed demonstrated the importance of the JEBAR effect, this term remains a diagnostic term, unlike prognostic terms, which are calculated by solving an equation and determining the unknowns. Thus, the JEBAR term provides explanations \textit{a posteriori}, once the ocean velocity field is known (or at least the density field), and therefore cannot be used directly in practice. For example, in \cite{greatbatch1991}, we can see how ocean gyres are recovered from the JEBAR diagnostic term.

In this work, we treat the interaction between topography and stratification in a prognostic manner, that is, we do not assume any term to be known. 
In other terms, we provide an anayltic derivation of a JEBAR-type effect.
A potential application of our results would be to obtain a closed formula describing the interaction between topography and stratified fluid flow, using a wall law derived from the effective boundary conditions on the principal-order solution, that could then be used as a parametrisation in coarse resolution ocean models. 

\subsection{Results and numerical illustrations} \label{sec:results}

Throughout the paper, we will need to switch between two sets of coordinates: the ``global'' ones, namely $(x_1,x_2,x_3)$, associated with the basis $(e_1,e_2,e_3)$ where $e_1$ denotes the eastward normalized vector, $e_2$ the northward one, and $e_3$ the vertical one; and the ``local'' coordinates, namely $(x,y,z)$, associated with the basis $(e_x,e_y,e_z)$ with $e_x=\cos \alpha e_1 + \sin \alpha e_3$, $e_y=e_2$, and $e_z=-\sin \alpha e_1 + \cos \alpha e_3$, see Figure \ref{fig:coordinates}.
The case where $\sin \alpha >0$ corresponds to an eastern boundary, and the case $\sin \alpha<0$ to a western boundary, which is our main focus here.
\begin{figure}[h]
\centering
\begin{tikzpicture}[scale=1.5,>=Stealth]
  \fill[pattern=north east lines, pattern color=gray!60] (0,2.5) -- (2,0) -- (0,0) -- cycle;
  \draw[thick] (0,2.5) -- (2,0) node[anchor=west] {$\partial \Omega$};
  \draw[->,thick,blue] (1.1,1.25) -- ++(1,0.8) node[pos=1.0, below right,blue] {$e_z$};
  \draw[->,thick,blue] (1.1,1.25) -- ++(0.8,-1) node[pos=1.0, above right,blue] {$e_x$};
  \draw[->,thick,red] (1.1,1.25) -- ++(1.3,0) node[pos=1.0, below right,red] {$e_1$};
  \draw[->,thick,red] (1.1,1.25) -- ++(0,1.3) node[pos=1.0, right,red] {$e_3$};
  \fill (2.95,2.4) circle (0.pt) node[below] {$\bigotimes\  \textcolor{red}{e_2} = \textcolor{blue}{e_y}$};
  \draw[-, thick] (1.55,1.25) to[out=270, in=50] (1.4,0.9) ;
  \fill (1.8,1.15) circle (0.pt) node[below] {$\alpha<0$};
\end{tikzpicture}
\caption{\textcolor{blue}{Local} and \textcolor{red}{global} coordinate systems in the case $\sin \alpha<0$ (western boundary).}
\label{fig:coordinates}
\end{figure}

Let us now introduce the main assumptions on the parameters and on the source term $f$.
\paragraph{Assumptions on the parameters.} In most sections of the paper, we will assume that
\be\label{hyp:parameters}\tag{H0}
\ba
\beta= \eps^{-a},\  \om= \eps^{-b},\  \nu_h=\eps^d,\ \nu_3=\eps^e,\ \delta=\eps,\\
\text{with }0< a<1,\quad e\geq d\geq 0,\quad  b \leq \frac{2a-d}{3},
\ea
\ee
where $\om$ is the time frequency of the forcing, see (H1) below.
Let us comment a little on these assumptions. 
The assumption $a<1$ means that $\eps \beta \ll 1$, and therefore that the $\beta$-plane approximation is legitimate (i.e. the sine of the latitude can be replaced by a local Taylor expansion). 
The assumptions $e\geq d\geq 0$ imply that the (rescaled) eddy diffusivities are small, and that the vertical diffusivity is smaller than the horizontal one, which is classical in an oceanographic context, see \cite{Pedlosky-GFD} and the derivation in \cref{appendix-derivation}.
The assumption $\delta=\eps$ on the aspect ratio $\delta$ is not essential. This choice stems from the formal derivation of the model (see Appendix \ref{appendix-derivation}). However, it will not affect the main bricks of the construction, namely the Munk and Ekman boundary layers.
In Section \ref{sec:heuristique} where we present the method to construct the approximate solution, and in Section \ref{sec:Ekman} where we compute the Ekman boundary layer, we have kept a general parameter $\delta$, in order to trace its influence on the construction.
In the different results below, we could have taken $\delta=\eps^{M/2}$ for an arbitrary $M\in \N$. The main impact lies in the iterative construction of the approximate solution, see \cref{rem:delta-m}.
Eventually, the assumption $b \leq \frac{2a-d}{3}$ stems from the analysis of Munk boundary layers (see \cref{lem:characteristic}). It can probably be relaxed into $b \leq \frac{3a-d}{4}$, although we did not perform the estimates on the whole approximate solution in this regime. We refer to Remark \ref{rem:regime} for further discussion.

We will always keep the parameters $\om, \beta$, etc. in the expressions without replacing them by powers of $\eps$, in order to keep the influence of each parameter as explicit as possible.

\paragraph{Assmptions on the source term.} We will assume that the source term $f=(f_h,0)$ satisfies the following assumptions:
\begin{enumerate}[(H1)]
    \item \textit{Time periodicity:} there exists a function $F$ such that $f(t,x,y,z)= \Re(e^{i\om t} F(x,y,z))$;
    \item \textit{Regularity:} $F\in H^q(\Om)$ for some sufficiently large $q\in \N$;
\item \textit{Exponential decay:} there exists $\gamma>0$ such that  for all $q\in \N$, for $(q_x,q_y,q_z)\in \N^3$ with $q_x+q_y+q_z\leq q$,
\[
\| \p_x^{q_x} \p_y^{q_y} \p_z^{q_z} F(\cdot, z)\|_{L^2(\R^2)}\leq C_q e^{-\gamma z};
\]

\item \textit{Spectral gap near zero:} there exists $Q>0$ large enough such that
\[
\int_0^\infty\int_{\R^2} |\xi_y|^{-Q} |\hat F(\xi_x, \xi_y, z)|^2\dd \xi_x \dd \xi_y \dd z <+\infty,
\]
where $\hat F$  denotes the Fourier transform of $F$ with respect to $\xi_x, \xi_y$.
    
\end{enumerate}

Under these assumptions, we can construct an approximate solution up to any order. We only give a rather vague statement here, and we will provide a more precise description in \cref{sec:proof-thm} (see \cref{lem:approx-high-order}):
\begin{proposition}
Let $N\geq 0$, $m\geq 0$ be arbitrary.
Assume that hypotheses (H0)-(H4) are satisfied, with sufficiently large exponents $q,Q$ depending on $a,b,d, e$, $m$ and $N$. Assume furthermore that $\sin \alpha <0$ (western boundary).

Then there exists an approximate solution $(\uapp,\rho^\mathrm{app}) \in H^m(\Om)^4$ of \eqref{NS-rho} with a source term $f+ g_\mathrm{rem}$ such that
\[
\| g_{\mathrm{rem}}\|_{L^\infty_t(L^2(\Om))} \leq \eps^N.
\]
Furthermore, $(\uapp,\rho^\mathrm{app})$ can be constructed explicitly in terms of the source term $f$, as an asymptotic expansion in powers of $\eps \beta=\eps^{1-a}\ll 1$. Each term in the asymptotic expansion is the sum of an interior term and of boundary layer terms.

\label{prop:sol-high-order-simplified}
    
\end{proposition}

We are now ready to state our main stability results:
\begin{theorem}
Assume that hypotheses (H0)-(H4) are satisfied, with sufficiently large exponents $q,Q$ depending on $a,b,d$, and $e$. Assume furthermore that $\sin \alpha <0$ (western boundary).
Let $(u,\rho)\in H^1(\Om)\times L^2(\Om)$ be a time periodic solution of \eqref{NS-rho}, with period $T=2\pi/\om$.

There exists an approximate solution of \eqref{NS-rho} of the form $(u^0 + \eps u^1, \rho^0 + \eps \rho^1)$, with $u^j, \rho^j \in H^1\cap L^\infty((0,T)\times\Om)$ defined in  \eqref{u-zero-h}, \eqref{d-3-p}, \eqref{systeme_interieur} and in \eqref{u1hh}, \eqref{eqsurcorrecteurh} and satisfying the estimates
\[
\| u^j_1\|_{H^1_{x,y}L^2_{t,z}} + \|u^j_3\|_{H^1_{x,y}L^2_{t,z}} \lesssim \beta^j  ,\qquad \| u^j_2\|_{H^1_{x,y}L^2_{t,z}} + \| \rho^j\|_{H^1_{x,y}L^2_{t,z}} \lesssim \beta^j \left( \frac{\beta}{\nu_h}\right)^{1/6}\quad\text{for }j=0,1,
\]
and such that
\[\ba
\forall k\in \{1,3\},\quad\| u_k-(u^0_k+ \eps u^1_k)\|_{H^1_{x,y}L^2_{t,z}} \lesssim &\; (\eps \beta)^2  ,\\ \| u_2-(u^0_2+ \eps u^1_2)\|_{H^1_{x,y}L^2_{t,z}} + \| \rho - (\rho^0 + \eps \rho^1)\|_{H^1_{x,y}L^2_{t,z}}\lesssim &\; (\eps \beta)^2 \left( \frac{\beta}{\nu_h}\right)^{1/6}.\ea
\]
The approximate solution can be computed explicitly and is the sum of an interior term and of boundary layer terms.
    \label{thm:stab-periodic}
\end{theorem}
\begin{remark}\label{rem:u0-geos-intro}
\begin{itemize}
    \item In fact, in view of \cref{prop:sol-high-order-simplified}, we could prove a stronger result.
    Indeed, it is also possible to prove the validity of a high order approximate solution,
     with an arbitrarily small remainder. 

\item We will give more details on the structure of the approximate solution in Section \ref{sec:heuristique}. Let us merely announce a couple of features: the main order term $u^0$ is geostrophic, i.e. $u^0= (\nabla_h^\bot p^0, 0) $ for some stream function $p^0$. However, unlike the unstratified case, the stream function $p^0$ is here 3D, which means that the velocity does depend on the vertical coordinate, and is not constrained by the isobaths. It consists of an interior term and a ``Munk'' boundary layer term, and it satisfies $u^0_{|\p \Om}=0$. The second term in the expansion, however, is not geostrophic. It is also the sum  of an interior term and a ``Munk'' boundary layer term, and it 
 is responsible for the detachment that is visible in the right part of \cref{fig:illustration}.

 We will see that the Ekman layer part is absent both from $u^0$, which is expected, and from $u^1$, which is less expected.

    \item In fact, when proving \cref{thm:stab-periodic}, we will obtain an error estimate in $L^1_t H^1(\Om)$. However, because of the presence of Ekman boundary layer terms, the size of $\| \p_z u^k\|_{L^2}$ for $k\geq 2$ is potentially very large, see \cref{lem:approx-high-order}. Hence we state our result in the space $H^1_{x,y}L^2_{t,z}$, in which the higher order terms are indeed negligible. 
\end{itemize}

\end{remark}

We will also prove a stability result for the Cauchy problem associated with \eqref{NS-rho}:
\begin{theorem}
Assume that assumptions (H0)-(H4) are satisfied, with exponents $q,Q$ depending on $a,b,d$, and $e$. Assume furthermore that $\sin \alpha<0$.

Let $(u^0 + \eps u^1, \rho^0 + \eps \rho^1)$ be the approximate solution constructed in \cref{thm:stab-periodic}.

Let $(u_{\mathrm{ini}}, \rho_{\mathrm{ini}})\in L^2(\Om)^2$ be such that $\Div u_{\mathrm{ini}}=0$, and let $(u,\rho)$ be the  weak solution of \eqref{NS-rho} with initial data $(u,\rho)(t=0)= (u_{\mathrm{ini}}, \rho_{\mathrm{ini}})$.
Assume that
\[
\| u_{\mathrm{ini},h} - (u^0_h + \eps u^1_h)_{|t=0}\|_{L^2} + \delta \|  u_{\mathrm{ini},3} - (u^0_3 + \eps u^1_3)_{|t=0}\|_{L^2} + \| \rho_{\mathrm{ini}} - (\rho^0 + \eps \rho^1)\|_{L^2} \lesssim (\eps \beta)^2.
\]
Then for all $t\geq 0$,
\[
\| u_h(t) - (u^0_h + \eps u^1_h)(t)\|_{L^2} + \delta \| u_3(t) - (u^0_3 + \eps u^1_3)(t)\|_{L^2} + \| \rho(t)- (\rho^0 + \eps \rho^1)(t)\|_{L^2} \lesssim (\eps \beta)^2 (1+t).
\]

    \label{thm:stab-Cauchy}
\end{theorem}

\begin{remark}
Note that because of the thin layer scaling, the approximation on $u_3$ from the energy estimate is degenerate. However one can retrieve an estimate on $u_3- (u^0_3 + \eps u^1_3)$ by using the divergence-free condition together with the estimate on $\nabla_h u_h$, for instance.
\end{remark}

\paragraph{Numerical experiments.} The framework used to perform simulations of the behaviour of the solution is the one presented above, with a twisted upper half-space, and a forcing that is periodic in time, oscillating and exponentially vanishing in $x_2$, and exponentially vanishing in $x_1$ and $x_3$.

Our main goal was to recover a behaviour similar to the one described in the paper by Zhang and Vallis \cite{zhang_role_2007}, that is, to observe a creation of positive vorticity where the western boundary current separates from the coast because of the effect of bottom pressure torque (and the associated bottom vortex stretching). Note that the bottom pressure torque is strongly linked to the JEBAR term, see for instance \cite{mertz_interpretations_1992}. This should translate into a separation southward, as explained in \cite{zhang_role_2007}.

\newpage
 
\begin{figure}[h]
    \centering
    \includegraphics[width=0.95\linewidth]{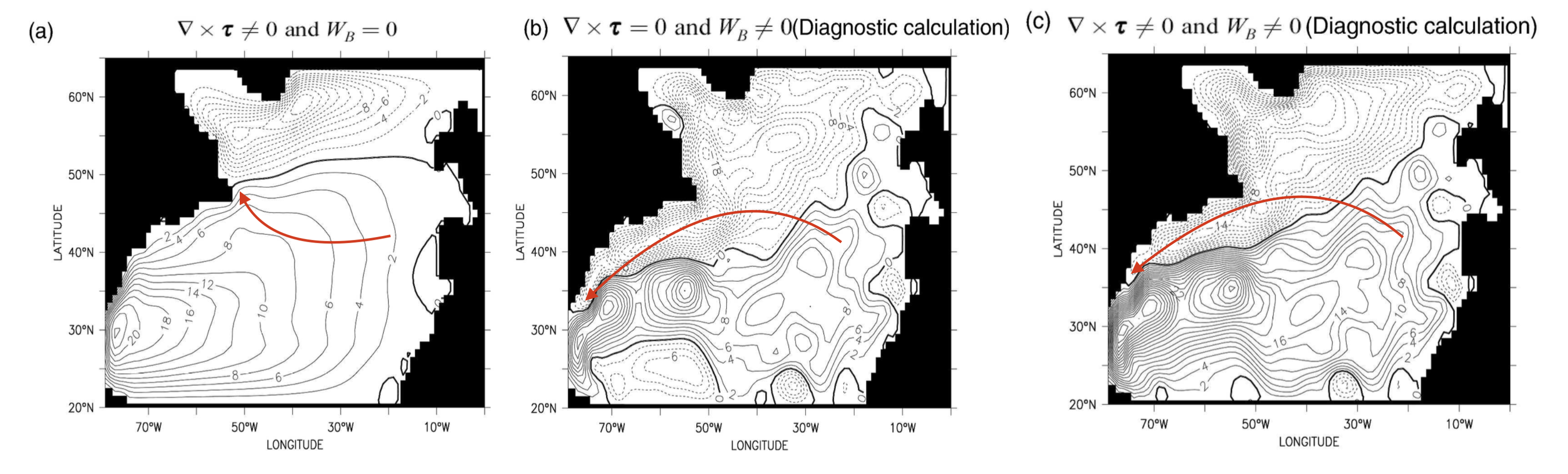}
    \caption{Stream function in the North Atlantic ocean computed with (a) only the wind stress, (b) only the diagnostic vertical velocity field, and (c) both (figure from \cite{zhang_role_2007} except the arrows).}
    \label{fig:zhang}
\end{figure}

This behaviour can be observed in Figure 12 of \cite{zhang_role_2007}, which we partly reproduce here in Figure \ref{fig:zhang} for the sake of clarity. The important feature is that the separation of the Western boundary current, indicated by the red arrows, occurs further south when we add the vertical velocity (48°N with no vertical velocity, against 32/35°N with vertical velocity).





Now, we present the behaviour of the solution in our study case.

\begin{figure}[h]
    \centering
    \begin{minipage}{0.48\textwidth}
        \centering
        \includegraphics[width=7cm,height=6cm]{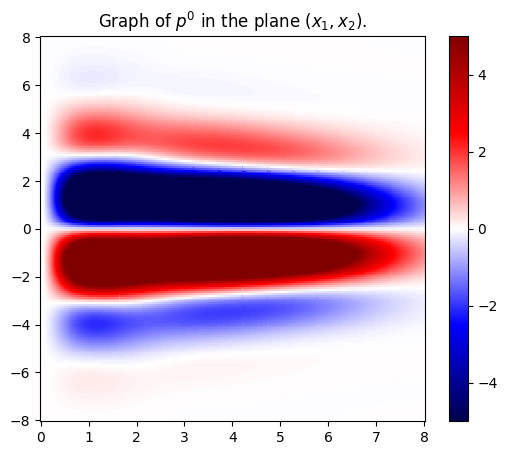}
    \end{minipage}
    \hfill
    \begin{minipage}{0.48\textwidth}
        \centering
    \includegraphics[width=7cm,height=6cm]{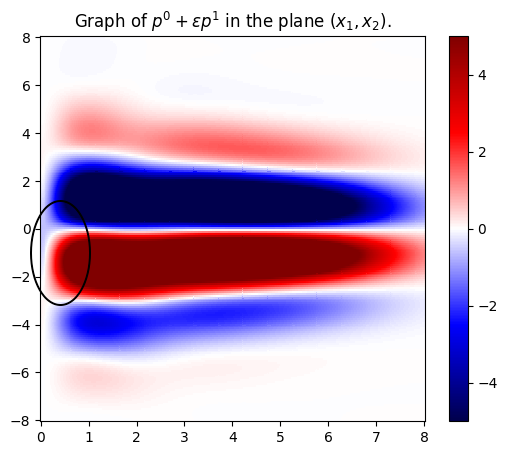}
    \end{minipage}
    \caption{Illustration of the role of topography in the behaviour of the stream function, in the $(x_1,x_2)$ plane, at $x_3$ fixed, plotted with Python (using exact terms explicited in the sequel of the paper). On the left, almost no difference is observed with the 2D case with no topography. On the right, the separation point of the western boundary current has been shifted southward, and the path is different.}
    \label{fig:illustration}
\end{figure}

The parameters used to obtain the figure \ref{fig:illustration} are the following. The angle $\alpha$ of the slope is such that $\cos^2(\alpha) = 0.2$ and $\sin^2(\alpha)= -0.8$. Moreover, the wind stress is the product of a gaussian in the $x_3$ coordinate, a gaussian as well in the $x_1$ coordinate, and a gaussian multiplied by a sine in the $x_2$ coordinate. The small parameters used are $\varepsilon = 0.005$, $\beta = \varepsilon^{-0.5}$, $\omega = \varepsilon^{-0.2}$, $\nu_h = \varepsilon^{0.2}$ and $\nu_3 = \frac{\nu_h}{10}$. A subsequent paper will explore in detail the role of the choice of parameters on the behaviour of the solution. 

Next, let us say a few words about the figure on the left to explain a little about the study configuration. In the $x_2$ direction, we have an oscillating and decreasing forcing term, which therefore produces four gyres, two dominant ones (the one between $x_2=0$ and $x_2=2$ is cyclonic, and the one symmetrical with respect to $x_2=0$ is therefore anticyclonic), and two that are evanescent and of less interest to us here. Furthermore, we can clearly see the structure of the solution as a superposition of an interior term that satisfies Sverdrup's equilibrium (see \eqref{svredrup} below), and a boundary layer term that corresponds in fact to the superposition of two boundary layers: the two usual Munk boundary layers. Finally, the main zero of the wind stress curl, which gives us the separation of the main boundary current, is located at $x_2=0$.

Now, regarding the figure on the right, there are two points to note. First, far from the boundary, the structure of the solution with no corrective terms appears to be preserved and Sverdrup's equilibrium remains respected. Second, it seems that the separation does occur further south, observing the main zero isoline of the stream function. This corresponds to the effect that we wanted to model, following \cite{zhang_role_2007}. 
The comprehensive study of the numerical model, which will be the subject of our future paper, will also provide an opportunity to explore this effect in further detail.






Finally, the separation pattern obtained in Figure \ref{fig:illustration} is very similar to the one obtained by NEMO simulations. We did simulation in a rectangular box, one with no topography, and two with topographies, and we saw that the impact of the topography was to move southward the separation point of the western boundary current, similarly to what \cite{zhang_role_2007} discribes. It can be understood through the fact that the flow will follow $\frac fh$ contours. 

\begin{figure}[h]
    \centering
    \includegraphics[width=14cm,height=5cm]{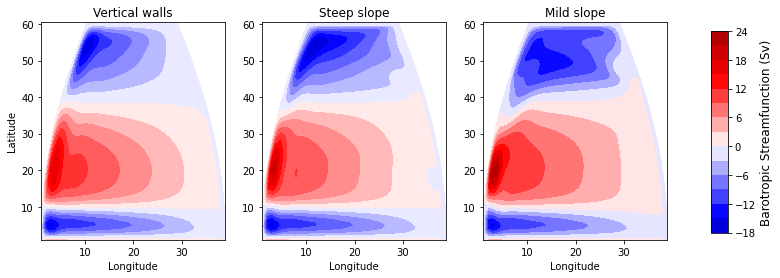}
    \caption{Barotropic stream function in three NEMO configurations, with the  slopes acting on the separation point latitude.}
    \label{fig:nemo}
\end{figure}

\subsection{Strategy of proof and comparison with previous works}\label{sec:strategy-intro}

Theorems \ref{thm:stab-periodic} and  \ref{thm:stab-Cauchy} strongly rely on the precise construction of an approximate solution to  equation \eqref{NS-rho}.
Several challenges are raised by such an equation. Its structure indicates that the solution should be geostrophic at  main order, i.e. of the form $(\nabla_h^\bot p,0)$.
However, because of the stratification, the solution must have a vertical component and therefore cannot be  purely geostrophic. 
Therefore, the approximate solution we will build comes with correctors of all orders, and has the structure \be \label{ansatz-intro-k} u = u^0 +\sum\limits_{k=1}^K \eps^k u^k,\ee
with $\nabla\cdot u^k = 0$ for each $k=0,...,K$.

Now, we explain the role and structure of the two types of terms $u^0$ and $u^k$, $k\geq 1$.

\begin{itemize}
\item The main order term $u^0$ is  the geostrophic component of the solution, i.e. $u^0=(\nabla_h^\bot p^0, 0)$, where  $p^0$ is the pressure at the main order, and plays the role of a stream function.
    The velocity $u^0$ (and so the function $p^0$) can be further decomposed into two parts: 
    \begin{itemize}
        \item an interior part denoted $u_{\mathrm{i}}^0$ that essentially satisfies Sverdrup equilibrium at  main order (this fact is proved in \cref{prop:QG3D-visc}). In particular, it does not see the boundary condition on western boundaries, i.e. when $\sin \alpha <0$.

        \item a boundary layer part of ``Munk type'', denoted $u_{\mathrm{M}}^0$ that allows us to satisfy two out of three boundary conditions on western boundaries. Note that the third boundary condition is satisfied because the velocity $u^0$ is purely horizontal.
    \end{itemize}
The overall construction of the geostrophic component is similar to previous works \cite{charve2005convergence,desjardins1998derivation,dalibard2018mathematical}; however we explore more thoroughly different regimes for the Munk boundary layers, identifying in particular regimes where the dissymetry between western and eastern boundaries may disappear (see \cref{lem:characteristic} and \cref{rem:dissymetry-EW}).

    \item The next order terms $\varepsilon^k  u^k$, for $k \geq 1$, are non-geostrophic corrections. Their presence ensure the conservation of mass (i.e. the evolution equation for $\rho$).
Each $u^k$ can be decomposed into the sum of an interior part, denoted $u_{\mathrm{i}}^k$, a Munk boundary layer part, denoted $u_{\mathrm{M}}^k$, and an Ekman part, denoted $u_{\mathrm{E}}^k$. 
The first two components (interior and Munk part) satisfy a quasi-geostrophic (QG) equation that will be derived in Section \ref{sec:heuristique}.
The role of the Ekman part is to ensure that the vertical component of the velocity vanishes on the boundary. 
It does not satisfy a QG equation, but solves the system \eqref{NS-rho} with no $\beta$ term.

 \end{itemize}

Note that the above decomposition is  classical (see for instance \cite{CDGG,desjardins1999homogeneous}). However, in most works, only the main order component is derived and analyzed.
Motivated by the description of western boundary current separation influenced by topography and stratification, we will need to push the expansion further. 
Indeed, the southward drift of the separation point is not visible at main order, but only within the first order correction $u^1$ (see \cref{fig:illustration}). 
Therefore our goal is truly to have an approximation result that validates the structure of the first order corrector, in the sense that $\| u-(u^0 + \eps u^1)\|=o(\eps \|u^1\|)$ in some suitable norm.
We encounter here a first difficulty: since equation \eqref{NS-rho} has a strong penalization due to the fast rotation, the thin layer effect, and the small diffusivity, we will need to build an approximate solution at a very high order in order to justify merely the very first terms of the expansion, see \cref{sec:proof-thm} for more details.

Another new feature of our work lies in the fact that the different boundary layers (i.e. the Munk and Ekman boundary layers) are all supported in the same region, namely in the vicinity of $z=0$.
In contrast, when the domain $\Om$ is of the form $(x_{\mathrm{W}}, x_\mathrm{E})\times \R\times (0,1)$, the Munk boundary layers are located in the vicinity of $x_1=x_\mathrm{W}$ and the Ekman layers near $x_3=0$ or $x_3=1$. Therefore, unless possibly in corners, the Munk and Ekman layers are not supported in the same region.
In the present study, understanding the interplay between the different boundary layers in the expansion is key.
In particular, let us highlight a phenomenon that was completely absent from previous studies on the subject.
In classical works on Ekman layers (see for instance \cite{grenier1997ekman,masmoudi2000ekman}), the Ekman boundary layer term is a linear combination of two decaying exponentials. 
As a consequence, there are two degrees of freedom associated with each Ekman layer.
In the present context, at first sight, we also find two degrees of freedom with the Ekman layer (see \cref{prop:Ekman}).
However, one of the two roots corresponds to a solution that is quasi-geostrophic, and therefore needs to be discarded (see \cref{rem:redundancy-Ekman} for more details).
Hence we only have one degree of freedom within the Ekman layer, and two within the Munk layer. These three degrees of freedom will ultimately allow us to ensure that the three components of the velocity vanish on the boundary. 
Let us emphasise another key feature of the single Ekman layer found in our case. The associated velocity vector is dominated by $e_x$, and so the Ekman boundary layer has a non-zero vertical component (i.e. along $e_3$) at first order. Therefore, the concept of Ekman pumping and the associated mechanisms are not fully relevant in our case, see Remark \ref{rem:structure-ekman-valide} for more details.


At last, let us mention that even though some recent works explore the behaviour of Ekman layers in the presence of a large topography \cite{chemin2024ekman}, the computation of Ekman layers in the presence of stratification and large topography had not been performed with system \eqref{NS-rho}. 
Although settings similar to \eqref{NS-rho} have been investigated 
in the physical literature (see for instance \cite{maccready_slippery_1993} in the case of small topography and \cite{garrett_boundary_1993} with diffusion only in the normal direction), the derivation we perform in the present paper seems completely new, together with the computation of the possible boundary layer sizes and profiles, and with the discussion around the two potential Ekman roots. 
Indeed the focus of \cite{maccready_slippery_1993,garrett_boundary_1993} is rather the understanding of the ``stopping phenomenon'', which is not our interest here since we analyze the linear response to an oscillatory forcing. 
 Our derivation differs rather strongly from the classical cases analyzed in \cite{grenier1997ekman,masmoudi2000ekman,CDGG}, and  the structure of the solution is also unusual. Indeed, in the $e_y$ direction, the viscous dissipation is balanced by the Coriolis force, as in classical Ekman layers. However, in the $e_x$ direction, the rotation does not play a role at main order: viscous dissipation is balanced by a combination of the pressure gradient and of the stratification (i.e. by the non-hydrostatic part of the pressure). 

\begin{remark}[Possible extensions]
As explained above, the present paper is a first step in the mathematical analysis of the effect of stratification and topography on the
separation of western boundary currents. However, model \eqref{NS-rho} is clearly an idealization, and our results could be generalized in many possible ways.
The most important extensions, both from the applied and the theoretical point of view, would be to add the nonlinear effects into the system, and to consider arbitrary (i.e. non flat) topographies.
Let us comment a little on these two perspectives. It is possible that part of the analysis of the present paper could be generalized to a nonlinear setting, as long as the nonlinearity is not too large, namely as long as the boundary layers remain linear at main order and the nonlinearity is only present in the interior of the flow, as in \cite{CDGG} for instance. 
In this case, each boundary layer term in the asymptotic expansion would satisfy a linear system with a source term depending in a nonlinear way on lower order correctors.
When the boundary layers become nonlinear, however, the situation is much more complicated. 
For instance, in some regimes, the Munk boundary layer in the presence of an advective term is akin to the Prandtl boundary layer, see \cite{dalibard2021existence,wang2018well}. In this regime, the analysis of the separation phenomenon and the presence of recirculating flows in the boundary layer may become very difficult to analyze mathematically, as the recent works \cite{dalibard2022linear,iyer2022reversal} demonstrate. 
Nonetheless, a recent study \cite{miller2025impact} describes from the physical point of view the impact of stratification on western boundary currents, using a 2-layer or 1.5-layer QG model with advection. Part of their discussion relies on the stability analysis of the inertial boundary layer, and therefore it would be very interesting to understand how the present analysis can be articulated with \cite{miller2025impact}, even at a formal level.

As for non-flat topographies, we believe that the present study can be used as a guideline to guess the structure and the interplay between the different boundary layers (of Ekman and Munk types) in situations in which the topography $\eta_B$ is non-flat, and in regions where its gradient is non-zero and bounded.
However, we expect that the major difficulties will be encountered when transitioning from regions where $1\lesssim | \nabla \eta_B| \lesssim 1$ 
(corresponding to a case where $\sin \alpha$ and $\cos \alpha$ are bounded away from zero in the present study)
to regions where $\nabla \eta_B=0$ (critical point, corresponding to $\sin \alpha=0$) or $|\nabla \eta_B|=\infty$ (vertical cliffs, corresponding to $\cos \alpha=0$).
In these two cases, the sizes of the boundary layers and their profiles change abruptly within a small horizontal region. This is related to the analysis of geostrophic degeneracy performed in \cite{dalibard2018mathematical} for a different problem (degeneracy of 2D Munk boundary  layers near the northern and southern coasts of an oceanic basin).
However, no general methodology exists for such problems, and hence we leave this issue aside in the present paper.

\end{remark}

\paragraph{Plan of the paper}
In Section \ref{sec:heuristique}, we derive the quasi-geostrophic (QG) equation, and give some general explanations on the structure of the solution. Then, in Section \ref{sec:QG}, we perform a detailed analysis of the QG equation and we provide a thorough description of the  geostrophic part of the solution. 
Afterward, we explain in Section \ref{sec:correcteurs} how to construct non-geostrophic correctors.  
Next, in Section \ref{sec:Ekman}, we study the two types of Ekman layers obtained under the $f$-plane hypothesis, that finally allow us to build a solution at any order and to prove the convergence results in \ref{sec:proof-thm}.

\paragraph{Notation} Throughout the paper, we will use the shorthands $\Delta_\nu$ for the diffusion operator $\nu_h \Delta_h + \nu_3 \p_3^2$, and $s$ for $\sin \alpha$, $c$ for $\cos \alpha$. 
We will also denote by $\nueff $ the viscosity coefficient in the direction normal to the boundary, i.e. $\nueff= \nu_h \sin^2\alpha + \nu_3 \cos^2\alpha$. The Fourier transform in the tangential variables $x,y$ {and in the time variable $t$} will be denoted either by $\widehat{~}$ or by $\cF$.


\section{Strategy for building an approximate solution} \label{sec:heuristique}

The purpose of this section is to present the approximate resolution method for the Boussinesq system given in the introduction, which is recalled here with $\delta^2 = \varepsilon$ for convenience:
\begin{subequations}\label{NS-rho-bis}
\begin{align}
\p_t u + \frac{1}{\varepsilon} (1 + \varepsilon \beta y)e_3 \wedge u + \frac{1}{\varepsilon}\begin{pmatrix}\nabla_h p\\ \varepsilon^{-1} \p_3 p \end{pmatrix}  - \nu_h \Delta_h u - \nu_3 \p_3^2 u=&\ \frac{1}{\varepsilon^2}\begin{pmatrix}
    0\\ 0\\ - \rho
\end{pmatrix}+\beta f\quad \text{in }\Om,   \label{L1NSrho}\\
\p_t \rho - \frac{1}{\varepsilon } u_3=&\ 0\quad \text{in }\Om,   \label{L2NSrho}\\
\Div u=&\ 0\quad \text{in }\Om,   \label{L3NSrho}\\
u\vert _{\p\Om}=&\ 0.    \label{L4NSrho}
\end{align} 
\end{subequations}

More precisely, we will present the articulation between the terms of different orders and of different natures. First, we will derive the QG equation and explain how to compute $u^0$. Then we will present how to compute the corrective term $u^1$. Finally we will show what are the role and the structure of the Ekman-type boundary layer term.

\subsection{Derivation of the quasi-geostrophic equation for the stream function}\label{sec:equation-fonction-courant}

To derive a quasi-geostrophic equation, we plug the first two terms of the asymptotic expansion \eqref{ansatz-intro-k} (\textit{i.e.} we take $K=1)$ into system \eqref{NS-rho}. This is a well-known method, see for example the textbook \cite{Vallis17}, which we recall  below for clarity.

Plugging the order $0$ term of \eqref{ansatz-intro-k} in \eqref{L1NSrho} ensures that, considering the order $\frac 1 \varepsilon$ for the horizontal part and $\frac{1}{\varepsilon\delta ^2}$ for the vertical part, the following equalities hold:\begin{subequations}
	\begin{align}
		(u_h^0)^\perp& = -\nabla_h p^0 \label{u-zero-h} \\
		\p_3 p^0 & = - \rho^0 . \label{d-3-p}
	\end{align}
\end{subequations}
The first relationship corresponds to geostrophic equilibrium, and the second is the hydrostatic approximation.

Geostrophic equilibrium \eqref{u-zero-h} ensures that $\nabla_h \cdot u^0_h = 0$. Therefore $\p_3 u_3^0 = 0$; now, since we look for finite energy solutions, we also have $\lim_{x_3\to \infty } u_3^0=0$. It thus follows that 
 $u_3^0$ vanishes everywhere, so $u^0$ is only horizontal. Furthermore, since $u^0$ has zero divergence, it derives from a stream function, that is clearly $p^0$. Notice that the vertical variation of $p^0$ is governed by $\rho^0$.

\begin{remark}
	We have justified that with an Ansatz such as \eqref{ansatz-intro-k}, the $u^0$ term is geostrophic. Hence, it corresponds in fact to $u_{\mathrm{i}}^0 + u_{\mathrm{M}}^0$, as explained in the introduction. 
    We will identify the interior and Munk terms below.
    Moreover, the $O(\varepsilon)$ terms will have a non-vanishing vertical velocity and introduce a non-geostrophic correction.
\end{remark}

Our next goal is  to obtain a closed equation for $p^0$. As is classical for penalization problems, this is achieved by taking into account the next order term in the asymptotic expansion and projecting the system on the space of constraints associated with the penalization, i.e. \eqref{u-zero-h}, \eqref{d-3-p}.
We start by plugging \eqref{ansatz-intro-k} into the horizontal part of \eqref{L1NSrho}, then apply the operator $\nabla_h^\perp\cdot(*)$, and noting $\zeta^0 = \nabla_h^\perp \cdot u_h^0$, we find the following equation for the terms of order $\varepsilon^0$ in the horizontal component of \eqref{L1NSrho}:
\be \label{zeta_h} \p_t\zeta^0 +\nabla_h^\perp \cdot\bigg( (u_h^1)^\perp + \beta y (u_h^0)^\perp \bigg) - \nu_h \Delta_h \zeta^0 - \nu_3 \p_3^2 \zeta^0 =\beta  \nabla_h^\perp \cdot f_h. \ee

At the same time, taking the vertical derivative   of \eqref{L2NSrho} at order $1$ (that is, $\varepsilon^0$), we find 
\be\label{dtd3rho}
\p_t\p_3 \rho^0 -  \p_3 u_3^1 =0
\ee
and we can now take the difference \eqref{zeta_h} - \eqref{dtd3rho}  to obtain

$$\underbrace{\p_t\zeta^0 - \p_t\p_3 \rho^0}_{\p_t \Delta p^0}+ \underbrace{\nabla_h^\perp \cdot  (u_h^1)^\perp + \p_3u^1_3}_{ = \nabla\cdot u^1 = 0}+ \underbrace{\nabla_h^\perp\cdot (\beta y (u_h^0)^\perp )}_{\beta \p_1p^0} - (\underbrace{\nu_h \Delta_h \zeta^0 + \nu_3 \p_3^2 \zeta^0}_{\Delta_\nu \Delta_h p^0}) =\beta  \nabla_h^\perp \cdot f_h, $$
where we have replaced $\rho^0$ by $-\p_3 p^0$ thanks to \eqref{d-3-p}. This gives the following equation on $p^0$ in the full domain $\Om$
\be \label{systeme_interieur} \boxed{\p_t\Delta p^0 +\beta \p_1 p^0 -\Delta_\nu \Delta_h p^0 = \beta\nabla_h^\perp \cdot f_h.} \ee
Variants of this equation, with or without the $\beta$ effect or an advection term, have already been studied abundantly in the literature, and we refer in particular to \cite{desjardins1998derivation,charve2005convergence} for its derivation, \cite{dalibard2018mathematical} for an analysis of the linear 2D case with a $\beta$ effect, \cite{puel2015global} for an analysis of the inviscid case (without the $\beta$ effect, and with a flat topography).
The study of this equation is performed in Section \ref{sec:QG}, but we present here some of the results that we will get to construct the solution of \eqref{systeme_interieur}.

First, considering the geometry of the problem, we can apply the Fourier transform in the tangential directions
$x$ and $y$ and in time. We get a simple ordinary differential equation in $z$, with coefficients depending on the Fourier variables and the physical parameters introduced. 
Classically, the solutions to this ODE are obtained as the sum of a particular solution to the non-homogeneous equation (with  source term $\beta \nabla_h^\perp \cdot f_h$), and a well-chosen solution to the homogeneous problem, that will allow us to satisfy some convenient boundary conditions. 

The interior term $u_{\mathrm{i}}^0=\nabla_h^\bot p_{\mathrm{i}}^0$ is then precisely identified as a particular solution to the non-homogeneous equation, defined as the convolution of the source term with   a Green function which we will compute explicitly (see Lemma \ref{def:Green-visc}). At main order, it coincides with the Sverdrup transport, i.e.
\be\label{svredrup} p_{\mathrm{i}}^0 \simeq -\int_{x_1}^\infty \nabla_h^\bot \cdot f_h(x_1',x_2,x_3) \dd x_1'.\ee

Moreover, the homogeneous part is a solution of $\p_t\Delta p +\beta \p_1 p -\Delta_\nu \Delta_h p = 0$.
After moving once again to Fourier variables in $x,y,t$,  we can look for solutions as $Ce^{-\mu z}$, with $C$ independent of $z$, and $\Re(\mu) >0$ (to ensure that the solutions have finite energy). Computations are detailed in \cref{prop:QG3D-visc}. 
Let us mention that when $\sin \alpha<0$ (western boundary), we find two admissible values of  $\mu$, so that the solution is expected to be the linear combination of two decaying exponentials $e^{-\mu_1 z}$ and $e^{-\mu_2 z}$, with $\Re(\mu_i)\gg 1$. 
This is consistent with the usual theory of Munk boundary layers near western coasts, see \cite{pedlosky1996ocean,desjardins1998derivation,dalibard2018mathematical}, which we generalize to the case with topography and stratification. 
The coefficients before the exponentials, depending on $\om$ and $\xi=(\xi_x,\xi_y)$ then ensure that $u^0$ vanishes on the boundary. More precisely, the structure of $u_{\mathrm{M}}^0$ is
$$\widehat{u_{\mathrm{M}}^0}= \widehat{\nabla_h^\perp p_{\mathrm{M}}^0} = \sum_{j\in \{1,2\}} c_j^0(\xi,\om)e^{-\mu_j z} \begin{pmatrix}
    -i\xi_y\\
    c i \xi_x+ s\mu_j
\end{pmatrix},$$
where we have used that $\p_1 (e^{i\xi_x x-\mu z})= (ci \xi_x+ s\mu)e^{i\xi_x x-\mu z} $. We choose $c_i^0$ so that
\be\label{def:cj0}
 \sum_{j\in \{1,2\}} c_j^0(\xi,\om) \begin{pmatrix}
    -i\xi_y\\
    c i \xi_x+ s\mu_j
\end{pmatrix}= - \widehat{u_{\mathrm{i},h}^0}\vert_{z=0}.
\ee
With this choice, we find that $c_i^0=O(1)$, and that $\|u_{\mathrm{M},2}^0\|_\infty=O(\mu)\gg 1 $. Therefore we retrieve the intensification of western boundary currents.

For further purposes, we note that $p^0=\p_n p^0=0$ on $\p \Om$.

\subsection{Determination of the first corrective term}

Let us now explain how to determine $u^1$ in terms of $u^0$. 
We use the same method as before, but we write \eqref{ansatz-intro-k} with $K=2$.
By plugging a higher order ansatz in \eqref{NS-rho}, we obtain a more exhaustive hierarchy of equations.
Computations similar to the ones of the previous 
subsection (that will be developed in Section \ref{sec:correcteurs}) give an equation for $u^1$ in terms of $u^0$, which writes 
\be\label{u1hh} u^1_h  = - \beta f_h^\perp - L^1 \nabla_h p^0  + \nabla_h^\perp p^1 ,\qquad u^1_3=\p_t \rho^0=-\p_t \p_3 p^0, \ee
where $L^1 v_h :  = \p_t v_h + \beta y v_h^\perp -\Delta_\nu v_h$ for $v_h= (v_1,v_2)$.

Then, plugging this equation into the hierarchy of equations deduced from the ansatz results in a scalar equation on $p^1$ with a similar structure to \eqref{systeme_interieur}, namely
\be \label{eqsurcorrecteurh}
(\p_t \Delta + \beta \p_1 -\Delta_\nu \Delta_h) p^1 = \nabla_h^\perp \cdot L^1 \bigg( \beta f_h^\perp + L^1 \nabla_h p^0 \bigg).
\ee

As before, we can determine a particular solution of this equation, which we denote by $\bar p^1$, by convoluting the right-hand side with the Green function associated with \eqref{systeme_interieur}. We denote by $\bar u^1$ the associated velocity. Since $p^0$ is a combination of interior and boundary layer terms, this structure will be transferred to $\bar p^1$, although we do not explicit this decomposition here.
Unfortunately the velocity $\bar u^1$ does not vanish on the boundary \textit{a priori}. Hence we need to construct further boundary layer correctors to lift the trace of $\bar u^1$.

We now make the following remark: although $\bar u^1_3=-\p_t \p_3 p^0$ is non zero, its trace vanishes on the boundary since  $p^0=\p_n p^0=0$ on $\p \Om$.
Thus $\bar u_1$ is purely horizontal at $z=0$. As a consequence the construction of boundary layers within $u^1$ is essentially the same as for $u^0$: the first order corrector will be of the form
\be \label{decompo-munk-1}
u^1_h=\bar u^1_h + \nabla_h^\bot \mathcal F^{-1}\left( \sum_{j\in \{1,2\}} c_j^1 e^{-\mu_j z}\right),\quad u^1_3=\bar u^1_3,
\ee
and the coefficients $c_j^1$ are determined by the conditions
\be \label{coeffc1}
 \sum_{j\in \{1,2\}} c_j^1(\xi,\om) \begin{pmatrix}
    -i\xi_y\\
    c i \xi_x+ s\mu_j
\end{pmatrix}= - \widehat{\bar u^1_h}\vert_{z=0}.
\ee
It follows that the Ekman layer at order one also vanishes since there is no vertical velocity to balance at $z=0$.

However, we emphasize that contrarily to $p^0$, $\p_3 p^1$ does not vanish on the boundary in general.
Hence, for $k\geq 2$, $\bar u^k_3$ does not vanish on the boundary, where we have generalized the notation above to higher order terms.
This vertical component cannot be lifted by Munk boundary layers, which remain purely horizontal. This is precisely where Ekman correctors come into play.

\subsection{Determination and role of Ekman boundary layers}

The construction proposed in the two previous subsections can be further generalised to higher order Ansatz, giving $u_{\mathrm{i}}^k$ depending on $u^l_{\mathrm{i}}$ for $0 \leq l \leq k-1$. Similarly, we can decompose the Munk-type term as \eqref{decompo-munk-1}, i.e. 
$$u_{\mathrm{M}}^k = \overline{u}_{\mathrm{M}}^k+ \begin{pmatrix}
    \nabla_h^\bot \mathcal F^{-1}\left( \sum_{j\in \{1,2\}} c_j^k e^{-\mu_j z}\right)\\0
\end{pmatrix},$$
with $\overline{u}_{\mathrm{M}}^k$ to be determined as a function of $\overline{u}_{\mathrm{M}}^l$ for $0\leq l \leq k-1$, and coefficients $c_j^k$ ensuring that the horizontal component of $u^k$ vanishes.

Ekman-type terms will also be written as a superposition of a term $\overline{u}_{\mathrm{E}}^k$ derived from the lower order Ekman terms, and a term that corresponds to a boundary layer with a coefficient 
and a structure to be determined: this is the subject of this subsection. 

In order to compute the structure of the Ekman boundary layer term, we assume that we can neglect the $\beta$ term in \eqref{NS-rho}, and we write the system in local coordinates. Then, we apply a Fourier transform in $x,y$ and $t$, and we look for wave solutions decaying like $\exp(-\lambda z)$ (so the $\p_x, \p_y$ become Fourier symbols and $\p_z$ becomes $-\lambda$). We obtain a linear system of the form $\mathcal A \mathsf{U}=0$, where $\mathcal A $ is a matrix depending on $\lambda$ (see \eqref{ekman_lin}),  and $\mathsf{U}$ is the wave amplitude of the solution in the different variables ($u_x,u_y,u_z,\rho,p)$. To allow this system to have non trivial solutions, we need to find values of $\lambda$ that cancel the determinant of $\mathcal A $. This results in two different sizes of Ekman boundary layers, with two eigenvectors denoted by $\mathsf{U}^1$ and $\mathsf U^2$ (provided these two eigenvectors are admissible, which is not the case).

We will prove in section \ref{sec:Ekman} that the eigenvector $\mathsf{U}^2$ is associated to a geostrophic velocity field at  main order, 
and therefore should not be taken into account. This was mentioned in \cref{sec:strategy-intro}, and will be detailed in \cref{sec:Ekman}. As a consequence, we are left with only three degrees of freedom with the boundary layer coefficients (two degrees of freedom for Munk layers, one for the remaining Ekman eigenvector ${\mathsf{\tilde U}^1}$, where the notation tilde stands for the projection on the three velocity components),  to match three scalar boundary conditions. 
Another important feature is that the vertical component of $\mathsf{\tilde U}^1$ is bounded away from zero, which will allow the Ekman boundary layer to lift the vertical component of the solution.
These two aspects strongly differ from classical works on Ekman layers above flat surfaces, in which the vertical component has size $\sqrt{\varepsilon\nu_3}$, \cite{Pedlosky-GFD,CDGG}.

Finally, the boundary condition writes 
\be\label{boundarycondition}\bigg(u_{\mathrm{i}}^k + \overline{u}_{\mathrm{M}}^k+ \begin{pmatrix}
    \nabla_h^\bot \mathcal F^{-1}\left( \sum_{j\in \{1,2\}} c_j^k e^{-\mu_j z}\right)\\0
\end{pmatrix} + \overline{u}_{\mathrm{E}}^k + \mathcal F^{-1}(c^k_{\mathrm{E}} e^{-\lambda_1 z} {\mathsf{ \tilde U}^1} )  \bigg)\vert_{\p \Omega}= 0. \ee
The role of the coefficient $c^k_{\mathrm{E}}$ in front of the Ekman boundary layer term is then to balance the vertical part of the solution, namely $\bigg(u_{\mathrm{i}}^k + \overline{u}_{\mathrm{M}}^k + \overline{u}_{\mathrm{E}}^k \bigg)\cdot e_3$.
Once this coefficient is determined, we compute  the $c_j^k$ coefficients by inverting a $2\times 2$ matrix as for the coefficients $c^0_j$ and $c^1_j$.

\begin{remark}\label{rem:fixed-point-VS-iteration}
	The Munk boundary layers have all the same size, that is, the roots associated to these boundary layers do not depend on $k$, and the same fact holds for Ekman layers.

    Therefore, there are two possible constructions for the boundary layers, one is iterative, and one relies on a fixed point method. 
    \begin{itemize}
        \item The iterative method is the one presented in the aforementioned construction. 
        If we want to compute the value of  the effective coefficient in front of the boundary layer (either  Ekman or Munk) for the entire solution at order $k$, we need to sum all $(c^l_j)_{l=0}^k$ for Munk boundary layers, and all $(c^l_{\mathrm{E}})_{l=2}^k$ for Ekman boundary layer. The advantages of this method are, on the one hand, that it is entirely explicit and, on the other hand, that it provides a practical solution for any order. It will be the point of view adopted in the whole paper.
        \item The fixed point method is slightly more abstract, and relies strongly on the fact that the shape of the boundary layers does not depend on the order to which they are computed. Heuristically,
        this would consist of constructing only the successive inner parts of the form $\bigg(u_{\mathrm{i}}^k + \overline{u}_{\mathrm{M}}^k + \overline{u}_{\mathrm{E}}^k \bigg)$, then finding the coefficient of each boundary layer once and for all. However, since each $\overline{u}^k$ depends on the coefficients $c_j, c_{\mathrm{E}}$, such a method is difficult to implement in practice.
    \end{itemize}
    
\end{remark}

The next sections are dedicated to the rigourous construction of approximate solutions, following the ideas given in this section.

\section{Analysis of the quasi-geostrophic equation}
\label{sec:QG}
The goal of this section is to analyze the  equation
\be
\label{eq:QG3D-visc}
i\om \Delta \psi + \beta \p_1 \psi - (\nu_h \Delta_h + \nu_3 \p_3^2) \Delta_h \psi= S \quad \text{in } \Om,
\ee
which corresponds to \eqref{systeme_interieur} to which a Fourier transform has been applied in time.
More precisely, we will construct solutions of \eqref{eq:QG3D-visc} and
 analyze their asymptotic behaviour in some parameter regimes slightly more general than the ones described in \eqref{hyp:parameters}. At this stage, we do not specify boundary conditions on $\p \Om$.
Our purpose is twofold:
\begin{itemize}
    \item We will define a Green function associated with \eqref{eq:QG3D-visc} and analyze its asymptotic behaviour as $\eps\to 0$ under the regime considered 
    (see \cref{def:Green-visc});
    \item We will also construct generic decaying solutions of the homogeneous equation associated with \eqref{eq:QG3D-visc}. We shall see that these solutions have a boundary layer behaviour: they are exponentially small outside a region of very small width, depending on $\eps$, and located in the vicinity of $\p\Om$.
    
\end{itemize}
Let us now give a bit more details about our strategy.
Since equation \eqref{eq:QG3D-visc} has constant coefficients and $\p \Om$ is flat, it is natural to apply the Fourier transform in the tangential variables $x$ and $y$, using the relationships $\p_1 = c\p_x-s\p_z$ and $\p_3 = c\p_z+s\p_x$. The equation then becomes an ODE in $z$, namely
\begin{equation}
    \label{ODE-QG3D-visc}
    \ba 
 i\om (\p_z^2 - |\xi|^2) \hat \psi &+ \beta (c i \xi_x- s \p_z) \hat \psi\\ &- \left(\nu_h \left((c i \xi_x- s \p_z)^2 - \xi_y^2\right) + \nu_3 (is\xi_x + c\p_z)^2 \right) \left((c i \xi_x- s \p_z)^2 - \xi_y^2\right) \hat \psi = \hS.
 \ea
    \end{equation}
Classically, solutions of this ODE are the sum of a specific solution and of the general solution to the associated homogeneous problem, which is in turn a linear combination of decaying exponentials.

We then have the following result on the roots of the associated characteristic equation:
\begin{lemma}
\label{lem:characteristic}
Let $\om \in \R$, $\beta, \nu_h, \nu_3\in (0, + \infty)$, $\xi \in \R^2$, $(\xi_y, \om)\neq (0,0)$. Consider the characteristic equation
\begin{equation}
    \label{eq:characteristic}
 i\om (\mu^2 - |\xi|^2) + \beta (c i \xi_x+ s \mu) - \left(\nu_h \left((c i \xi_x+ s \mu)^2 - \xi_y^2\right) + \nu_3 (is\xi_x - c\mu)^2 \right) \left((c i \xi_x+ s \mu)^2 - \xi_y^2\right)  =0.
\end{equation}
Assume  that $\sin \alpha<0$ (western boundary) and  $\xi_x\neq 0$.
The following results hold:

\begin{enumerate}
    \item Equation \eqref{eq:characteristic} has two complex roots with positive real parts, denoted by $\mu^+_1$ and $\mu^+_2$, and two complex roots with negative real parts, denoted by  $\mu^-_1$ and $\mu^-_2$. Without loss of generality, we assume that $\Re(\mu^\pm_1)\geq \Re(\mu^\pm_2)$.

\item Assume that\footnote{Recall that $\nueff= \nu_h \sin^2\alpha + \nu_3 \cos^2\alpha$.} $\nueff\ll \beta$, $|\om|\lesssim \beta^{2/3} \nueff^{1/3}$ and $| \xi | \ll \beta^{1/3} \nueff^{-1/3}$.  Then for $j=1,2$,
\[
\mu^+_j\sim \left(\frac{\beta}{\nueff}\right)^{1/3} M^+_j,\ \mu^-_2 \sim \left(\frac{\beta}{\nueff}\right)^{1/3} M^-_2,
\]
where $M^+_1, M^+_2$ and $M^-_2$ are the three roots of the polynomial
\be\label{eq:Munk-rescaled}
i \frac{\om}{\beta^{2/3} \nueff^{1/3}} M + s  - s^2 M^3=0,
\ee
such that $\Re (M_2^-)< 0$, while
\[
\mu^-_1 \sim - i \frac{c\xi_x}{s}.
\]

\end{enumerate}

\end{lemma}
Lemma \ref{lem:characteristic} will be proved in Section \ref{sec:proof-characteristic}.

\begin{remark}
   The regime depicted in the second item of the Lemma corresponds to Assumption \eqref{hyp:parameters}, but actually we can push a little further the assumptions on the parameters and prove the following result:
Assume that $\nueff\ll \beta$ and that $ \beta^{2/3} \nueff^{1/3} \ll | \om | \ll \beta^{3/4} \nueff^{1/4}$, $| \xi| \ll \beta^3 \nueff |\om|^{-4}$.
Then
\[
\mu^+_1\sim \left| \frac{\om}{\nueff s^2}\right|^{1/2} e^{i\sgn \om \pi/4},\quad \mu^-_2\sim - \mu^+_1
\]
while $\mu^-_1\sim -ic\xi_x /s$ as above and
\[
\mu^+_2 \sim - \frac{\beta s}{i \om},\quad \Re( \mu^+_2) \sim - \frac{\nueff s^5 \beta^3}{\om^4}\gg 1.
\]

The proof of this statement is left to the reader, since it follows from the same arguments as \cref{lem:characteristic} but will not be used in the rest of the paper.

\label{rem:regime}   
\end{remark}

Using the definition of the roots $(\mu^\pm_i)_{i=1,2}$, we define the Green function associated with equation \eqref{eq:QG3D-visc}:
\begin{lemma}[Green function]
    \label{def:Green-visc}
Let $\om \in \R$, $\beta, \nu_h, \nu_3\in (0, + \infty)$ satisfying the assumptions of \cref{lem:characteristic}.
Define the distribution $G$ as $G(\cdot, z)=\mathcal F^{-1}(\widehat{ G}(\xi, z) )$, where
\[
\cG(\xi, z)=
\begin{cases}
\sum_{i=1,2} C_i^+ \exp(-\mu_i^+ z)\text{ for }z>0,\\
\sum_{i=1,2} C_i^- \exp(-\mu_i^- z)\text{ for }z<0,
\end{cases}
\]
with
\[
C_j^\pm = \mp \frac{1}{\nueff s^2}\prod_{\mu \in \mathcal M\setminus \{\mu_j^\pm\}} \frac{1}{\mu - \mu_j^\pm }. 
\]

    Then $ G$ is the Green function associated with associated with equation \eqref{eq:QG3D-visc}: in the sense of distributions,
    \[
    i \om \Delta G + \beta \p_1 G -(\nu_h \Delta_h + \nu_3 \p_3^2 )\Delta_h G= \delta(z=0).
    \]

\end{lemma}

We are now ready to state our result on equation \eqref{eq:QG3D-visc}:

\begin{proposition}[Solutions of the 3D viscous QG equation]
Let $\om \in \R$, $\beta, \nu_h, \nu_3\in (0, + \infty)$ such that $| \om| \lesssim \beta^{2/3} \nueff^{1/3}$.
Let $S\in L^2(\R^2_{x,y}, L^1_z)$. We assume that there exists $R>0$ such that 
$
R\ll \left(\frac{\beta}{\nueff}\right)^{1/3} 
$
and such that
$\hS(\xi, z)=0$ for a.e. $|\xi|\geq R$, $z>0$.

\begin{enumerate}
    \item Define $\psi_S:= G\ast S$, i.e.
    \be\label{def:psif}
    \widehat{ \psi}_S(\xi,z)=\int_0^\infty \widehat{ G}(\xi, z-z') \hS(\xi, z')\dd z'.
    \ee
Then $\psi_S\in L^2(\R^2_{x,y}, C_0(0, +\infty))$ and $\psi_S$ is a solution of \eqref{eq:QG3D-visc} in the sense of distributions, with
\[
\| \psi_S\|_{L^2(\R^2, L^\infty(\R_+))} \lesssim \frac{1}{\beta} \| S\|_{L^2(\R^2, L^1(\R_+))},
\]
and more generally, for any $s\geq 0$,
\[
\| \psi_S\|_{H^s(\R^2, L^\infty(\R_+))} \lesssim \frac{1}{\beta} \| S\|_{H^s(\R^2, L^1(\R_+))}.
\]

Furthermore, if $\psi\in  L^2(\R^2_{x,y}, C^1_0(0, +\infty))$ is any solution of \eqref{eq:QG3D-visc} such that $\widehat{ \psi}(\xi, z)=0$ for $|\xi| \geq R$, then there exist $c_i\in L^2(\R^2)$ such that
\[
\widehat{ \psi} (\xi, z)= \widehat{ \psi}_S (\xi, z) + \sum_{i\in \{1,2\}} c_i (\xi) \exp(-\mu_i^+ z).
\]

\item Assume that  \eqref{hyp:parameters} is satisfied.
Then
\[
\lim_{\eps \to 0} \beta( \psi_S- \overline{\psi}_S)=0,
\]
where
\[
\beta \p_1 \overline{\psi}_S = S, \quad \lim_{z\to + \infty }\overline{\psi}_S=0.
\]

\end{enumerate}

    \label{prop:QG3D-visc}
\end{proposition}

\begin{remark}\label{rem:dissymetry-EW}
\cref{lem:characteristic} and \cref{prop:QG3D-visc} generalize to a 3D setting well-known results for the Munk model, which is a 2D version of \eqref{eq:QG3D-visc}, see for instance \cite{desjardins1999homogeneous,dalibard2018mathematical}. Note that $\overline{\psi}_S$ satisfies the so-called Sverdrup equilibrium.
In our 3D setting, the relevant viscosity parameter for boundary layer theory is $\nueff=\nu_h s^2 + \nu_3 c^2$, which is the diffusion coefficient in the direction which is normal to the boundary.

When $s>0$ (eastern boundaries), an analysis similar to the proof of \cref{lem:characteristic} shows that 
$\mu_2^+ \sim -ic \xi_x/s$. In the regime considered, there is only one root with large and positive real part, namely $\mu_1^+$.
As a consequence, in this case, we still have $\beta (G\ast S- \overline{\psi}_S)\to 0$, but $ \overline{\psi}_S$ satisfies $\beta \p_1 \overline{\psi}_S = S$ and $\overline{\psi}_S(z=0)=0$.
We therefore retrieve the well-known dissymetry between western and eastern boundaries, which explains the intensification of boundary currents near the western coasts of oceanic basins, since $u_h=\nabla_h^\bot \psi$.

We emphasize however that when $|\om| \gg \beta^{3/4}\nueff^{1/4} $ and $\sin \alpha<0$, the real part of $\mu_2^+$ becomes small. In this regime, the dissymetry between western and eastern coasts may dissapear.
To illustrate this, let us look at the polynomial given by equation \eqref{eq:Munk-rescaled}. First note that the polynomial has the form $M^3 -ipM +q$ with 
$p=\om/(\beta^{2/3} \nueff^{1/3}s^2)$ and $q=-1/s$, so that
$p,q>0$ if $\sin(\alpha)<0$ and $\omega>0$. Therefore, we can perform the change of variable $M\mapsto M p^\frac12 := \widetilde M$, which leads to the study of roots of the polynomial $\widetilde{M}^3-i\widetilde M + \tilde q$, with $\tilde q=-s^2 \beta \nueff^{1/2}/\omega^{3/2}$. Thus, we can plot the roots as in \cref{fig:roots}.

\begin{figure}[h]
    \centering
    \includegraphics[width=0.45\textwidth]{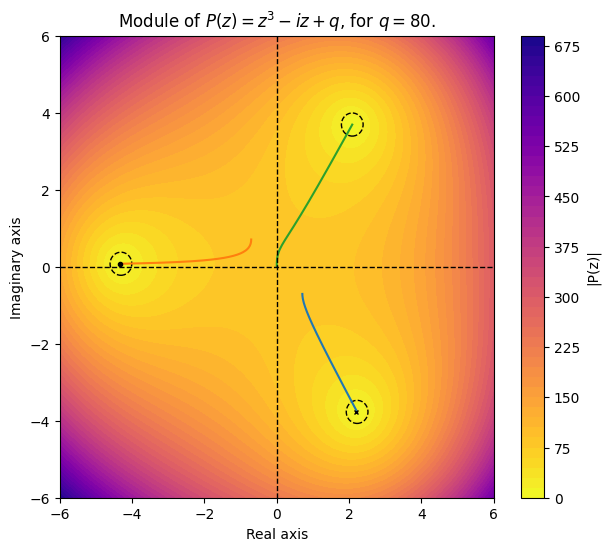}
    \hfill
    \includegraphics[width=0.45\textwidth]{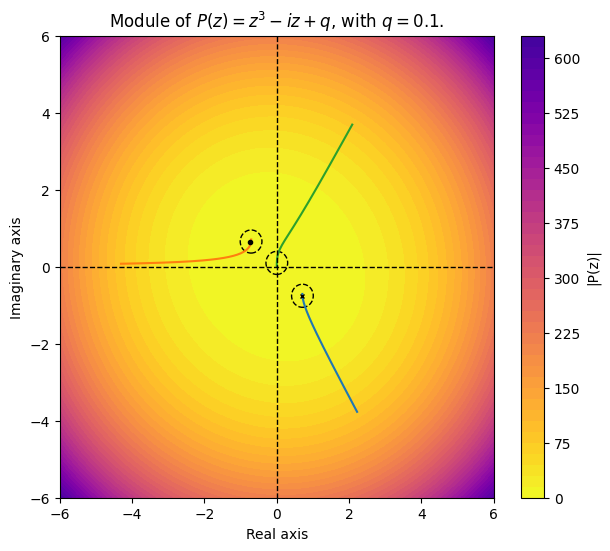}
    \caption{Module of the studied polynomial, and curves described by the three roots as $\tilde q$ vary. The three roots are plotted for two different values of $\tilde q$. The orange curve corresponds to $M_2^-(\tilde q)$, the blue one to $M_1^+(\tilde q)$ and the green one to $M_2^+(\tilde q)$.}\label{fig:roots}
\end{figure}
On the left of \cref{fig:roots}, we are in a regime where $\vert \om \vert \ll \beta^\frac23 \nueff^\frac13$, that is point 2 of Lemma 3.1, and we recoover the usual Munk roots that give two boundary layers on West boundary, and therefore we have a dissymetry between East boundary and West boundary. On the right, we are no longer in the regimes covered by Lemma \ref{lem:characteristic} or \cref{rem:regime} (we have $\beta^\frac34 \nueff^\frac14 \lesssim \vert \om \vert$), and we find that the real part of $\mu_2^+$ becomes small.

\end{remark}

\subsection{Proof of Lemma \texorpdfstring{\ref{lem:characteristic}}{on characteristic roots}}
\label{sec:proof-characteristic}

Let $\Lambda=(\om, \xi_x,\xi_y, \alpha, \beta, \nu_h, \nu_3)\in \R^7$ such that $\nu_h,\nu_3>0$, $\sin \alpha\neq 0$, $(\xi_y,\om)\neq (0,0)$. The polynomial equation \eqref{eq:characteristic}, which we write as $P(\Lambda ; \mu)=0$ always has four complex roots $\mu_i(\Lambda)$. 
If $P(\Lambda_0,\mu_0)=0$, and if $\Lambda$ lies in the neighborhood of $\Lambda_0$, we may look for a root $\mu$ in the neighborhood of $\mu_0$  with the help of the implicit function theorem. The latter can be applied as long as $d_\mu P(\Lambda_0,\mu_0)$ is invertible, i.e. as long as $\mu_0$ is a simple root of $P(\Lambda_0, \cdot)$.
This proves that
the roots $\mu_i(\Lambda)$ are $C^\infty$ in $\Lambda$ on the set $\{\Lambda\in \R^7, \mu_i(\Lambda)\neq \mu_j(\Lambda) \text{ for } i\neq j\}$.
Classical arguments also ensure that the curves $\Lambda\mapsto \mu_j(\Lambda)$ are continuous. Hence  the number of roots with positive real part is independent of $\Lambda$: indeed, if this number depended on $\Lambda$, then by the intermediate value theorem there would exist $\Lambda_0\in \R^7$ and $j\in \{1,\cdots 4\}$ such that $\mu_j(\Lambda_0)\in i\R$. In this case, taking the real and imaginary parts of \eqref{eq:characteristic}, we would have simultaneoulsy
\begin{align*}
 i\om (\mu_j^2 - |\xi|^2) + \beta (c i \xi_x+ s \mu_j) &=0,\\
\left(\nu_h \left((c i \xi_x+ s \mu_j)^2 - \xi_y^2\right) + \nu_3 (is\xi_x - c\mu_j)^2 \right) \left((c i \xi_x+ s \mu_j)^2 - \xi_y^2\right)&=0.
    \end{align*}
The second line implies that one of the two factors is zero. Since each factor is a sum of (opposites of) squares of real numbers, this only occurs when $\xi=0$, which is excluded, or when $\xi_y=0$ and $ci\xi_x + s\mu_j=0$. In the latter case, plugging the equality into the first line, we obtain $\om=0$. Since we have assumed that $(\xi_y,\om)\neq (0,0)$, there are no pure imaginary roots.

There remains to prove that there are two roots with positive real parts and two with negative real parts. 
To that end, in view of the above discussion, it suffices to count the number of roots with positive real parts for specific values of the parameters. Take $\nu_3=\nu_h=1$ and $\beta=\om=0$. Then it can be easily checked that the roots are given by
\[
\left\{ \frac{\xi_y - i c\xi_x}{\sin \alpha}, \frac{-\xi_y - i c\xi_x}{\sin \alpha}, | \xi|, - | \xi|\right\}.
\]
The first point of \cref{lem:characteristic} follows.

Let us now turn towards the next two assertions. Setting first
\[
\mu= \left(\frac{\beta }{\nueff}\right)^{1/3} M, 
\]
we find that $\mu$ is a solution of \eqref{eq:characteristic} if and only if $M$ is a solution of $Q(\Lambda, M)=0$, where
\begin{align*}
Q(\Lambda, M) &= M^4 + a_3 M^3 + a_2 M^2 + a_1 M + a_0\\
&=\widetilde{Q}(\Lambda, M) + Q_{\mathrm{rem}}(\Lambda, M),
\end{align*}
where the coefficients $a_i$ are such that
\begin{align*}
a_3= O\left(\xi_x\nueff^{1/3} \beta^{-1/3}\right), \quad a_2= - \frac{i\om}{\beta^{2/3} \nueff^{1/3} s^2} + O\left(|\xi|^2\nueff^{2/3} \beta^{-2/3}\right),\\
a_1=-\frac{1}{s} + O\left(|\xi|^3\nueff \beta^{-1}\right),\quad a_0= -i c\xi_x s^{-2} (\nueff/\beta)^{1/3}+ O\left(\om \nueff^{1/3} \beta^{-4/3} |\xi|^2 + |\xi|^4 \nueff^{4/3} \beta^{-4/3}\right),   
\end{align*}
and
\[
\widetilde{Q}(\Lambda, M) = M^4   - \frac{i\om}{\beta^{2/3} \nueff^{1/3}s^2} M^2 -\frac{1}{s} M .
\]
Note that in the regime $\nueff\ll \beta$ and $|\om| \lesssim \beta^{2/3} \nueff^{1/3}\ll \beta$, $| \xi| \ll (\beta/\nueff)^{1/3}$, all the coefficients of the polynomial $ Q_{\mathrm{rem}}(\Lambda,M)$ are $o(1)$, while the ones of the polynomial $\widetilde{Q}(\Lambda, M)$ are $O(1)$. Now, let $\widetilde{M}$ be a root of $\widetilde{Q}$. We look for a root of $Q$ in the form $M=\widetilde{M} + \eta$, with $\eta \ll 1$ (note that we include the possibility that $\widetilde{M}=0$). 
At main order, this gives
\[
\underbrace{\mathrm{d}_M \widetilde{Q}(\Lambda, \widetilde{M}) (\eta)}_{=\left( 4 \widetilde{M}^3 -  \frac{2i\om}{\beta^{2/3} \nueff^{1/3}s^2} \widetilde{M} - \frac{1}{s}\right) \eta } +\  Q_{\mathrm{rem}}(\Lambda, \widetilde{M})=o(\eta).
\]
It can be checked that $\mathrm{d}_M \widetilde{Q}(\Lambda, \widetilde{M}) $ is invertible at each root $\widetilde{M}$ of $\widetilde{Q}$.
A variant of the implicit function theorem  then implies that $Q(\Lambda, \cdot)$ has a root $M$ in the vicinity of $\widetilde{M}$, with $\eta =M-\widetilde{M}= O(Q_{\mathrm{rem}}(\Lambda, \widetilde{M}) )$.
This provides the desired equivalent for $\mu_i^+$ and $\mu_2^-$, which correspond to the non-zero roots of $\widetilde{Q}$. For $\mu_1^-$, we take $\widetilde{M}=0$, and we observe that
\[
\mathrm{d}_M \widetilde{Q}(\Lambda, \widetilde{M}) (\eta)=-\frac{\eta}{s}.
\]
It follows that when $\xi_x\neq 0$ and $\xi_x/|\xi| \gtrsim 1$,
\[
\mu_1^-\sim  \left(\frac{\beta }{\nueff}\right)^{1/3}s  Q_{\mathrm{rem}}(\Lambda,0)\sim -\frac{ic\xi_x}{s}.
\]
\qed

\begin{remark}
In the two regimes described in \cref{lem:characteristic} and \cref{rem:regime}, the eigenvalues $\mu_j^\pm$ are well-separated. 
In the regime $|\om| \propto \beta^{2/3} \nueff^{1/3}$ and $|\xi| \propto (\beta/\nueff)^{1/3} $, it does not seem obvious to prove that the eigenvalues do not cross, although numerical simulations seem to indicate that they remain simple. Since we wish to exclude such pathological situations, we will always assume that $\vert\xi\vert $ lies in a compact set (whose size depends on $\varepsilon$).
\label{rem:eigenvalue-crossing}
\end{remark}

\subsection{Construction of the Green function}
\label{ssec:Green}
We look for the function $\cG$ in the form
\[
\cG(\xi, z)=
\begin{cases}
\sum_{i=1,2} C_i^+ \exp(-\mu_i^+ z)\text{ for }z>0,\\
\sum_{i=1,2} C_i^- \exp(-\mu_i^- z)\text{ for }z<0.
\end{cases}
\]
Since $\cG$ should satisfy the ODE \eqref{ODE-QG3D-visc} with $\hS=\delta(z=0)$, this provides the following jump conditions:
\[
[\p_z^k \cG]_{|z=0}=0\text{ for }k=0,1,2,\quad [\p_z^3 \cG]_{|z=0}= -\frac{1}{s^2\nueff}.
\]

The coefficients $C_i^\pm$ are therefore determined by the linear system
\begin{equation}
    \label{systemegreen}
    \begin{pmatrix}
        1&-1&1&-1\\ 
        \mu_1^+ &- \mu_1^- &\mu_2^+ &- \mu_2^-  \\
        (\mu_1^+)^2 &-(\mu_1^-)^2  &(\mu_2^+)^2 &- (\mu_2^-)^2 \\
        (\mu_1^+)^3 &- (\mu_1^-)^3 &(\mu_2^+)^3  &-(\mu_2^-)^3
    \end{pmatrix} 
    \begin{pmatrix}
        C_1^+\\C_1^-\\C_2^+\\C_2^-
    \end{pmatrix}
    =
    \begin{pmatrix}
        0\\0\\0\\\frac{1}{\nueff s^2}
    \end{pmatrix}.
\end{equation}
Note that if two of the eigenvalues coincide for some value of the parameters $\nu_h, \nu_3, \om, \xi$, the system is not invertible. Therefore, throughout this paragraph, we will assume that the parameters satisfy the assumptions of \cref{lem:characteristic} (see also \cref{rem:eigenvalue-crossing}).

The matrix in the left-hand side of \eqref{systemegreen} is the (transpose of the) matrix of the application
$$\Phi : \begin{cases}
    \bC_3[X] \longrightarrow \bC^4\\ 
    P\ \ \ \quad \mapsto \  \ (P(\mu_1^+), -P(\mu_1^-),P(\mu_2^+),-P(\mu_2^-))
\end{cases}$$
in the canonical bases of $ \bC_3[X] $ and $\bC^4$. 
Hence we need to invert the application $\Phi$. To that end, we rely on Lagrange interpolation polynomials.  Let $\mathcal M:=\{\mu_1^+, \mu_1^-, \mu_2^+, \mu_2^-\}$.
For $j\in \{1,2\}$, we set
\[
P_j^\pm(X)= \prod_{\mu \in \mathcal M\setminus \{\mu_j^\pm\}} \frac{X-\mu}{(\mu_j^\pm - \mu)}
\]
so that $P_j^\pm\in\bC_3[X] $, $P_j^\pm(\mu)=\delta_{\mu, \mu_j^\pm} $ for $\mu\in \mathcal M$.
Then $\Phi^{-1}$ is the mapping
\[
\Phi^{-1}: 
\begin{array}{ccc}
    \bC^4 &\longrightarrow& \bC_3[X] \\ 
    (a_1^+, a_1^-, a_2^+, a_2^-)& \mapsto &\sum_{j,\pm} \pm a_j^\pm P_j^\pm.
\end{array}
\]
In order to find the invert of the matrix in the left-hand side of \eqref{systemegreen}, it suffices to write the (transpose of the) matrix of $\Phi^{-1}$ in the canonical bases of $\bC^4$ and $ \bC_3[X] $.
More precisely, given the form of the right-hand side of \eqref{systemegreen}, we are only interested in the last column of this matrix, which consists of the leading order coefficients of $(\pm P_j^{\pm})$. Eventually, we infer that
\[
C_j^\pm = \pm \frac{1}{\nueff s^2}\prod_{\mu \in \mathcal M\setminus \{\mu_j^\pm\}} \frac{1}{(\mu_j^\pm - \mu)}. 
\]
Hence we obtain the formula announced in \cref{def:Green-visc}.

Let us also derive some estimates on the coefficients $C_j^\pm$ in the  regime $| \om| \ll \beta^{2/3} \nueff^{1/3}$, $|\xi| \ll (\beta/\nueff)^{1/3}$.
The roots
$\mu_1^+, \mu_2^+, \mu_2^-$ are such that
\[
|\mu| \gtrsim \left(\frac{\beta}{\nueff}\right)^{1/3},\quad |\mu-\mu'| \gtrsim \left(\frac{\beta}{\nueff}\right)^{1/3}\text{ for }\mu\neq \mu'\in \{ \mu_1^+,\mu_2^-,\mu_2^+ \}. 
\]
It follows that $|C_j^\pm| \lesssim \beta^{-1}$. 
We also note that the coefficients $C_j^\pm$ are independent of $\xi$ at main order.

\subsection{Proof of Proposition \texorpdfstring{\ref{prop:QG3D-visc}}{3.3}}

It follows from the previous Section that, in the regime considered here,
\[
| \cG (\xi, z)| \lesssim \beta^{-1}.
\]
Therefore, for a.e. $\xi, z$,
\[
| \widehat{\psi}_S(\xi, z)| \lesssim \beta^{-1} \int_0^\infty | \hS(\xi, z')| \dd z'.
\]
The $L^2(L^\infty)$ estimates on $\psi_S$ follow, together with the $H^s(L^\infty)$ estimates after a mere mutliplication by the factor $\langle \xi\rangle^s$. 
The property $\lim_{z\to \infty} \psi_S(\cdot, z)=0$ is a straightforward consequence of the Lebesgue dominated convergence theorem. 
Noticing that $| \p_z \cG|\lesssim \beta^{-1} \max\limits_{j,\pm} | \mu_j^\pm|$, we also find that $\p_z \psi_S \in L^2(\R^2_{x,y}, L^\infty_z)$.

Now, let $\psi\in L^2(\R^2, C_0(0, +\infty))$ be a solution of \eqref{eq:QG3D-visc} in the sense of distributions. Passing to Fourier variables in $\mathcal S'(\R^2)$, we infer that $\widehat{\psi}-\widehat{\psi}_S$ is a solution of \eqref{ODE-QG3D-visc}, with a null right-hand side. We then infer from classical ODE theory that
\[
(\widehat{\psi}-\widehat{\psi}_S)(\xi, z)=\sum\limits_{j=1,2} c_j(\xi) \exp(-\mu_j^+(\xi) z).
\]
Since $(\widehat{\psi}-\widehat{\psi}_S)\vert_{z=0}\in L^2(\R^2)$ and $\p_z (\widehat{\psi}-\widehat{\psi}_S)\vert_{z=0}\in L^2(\R^2)$, we deduce that $c_1, c_2\in L^2(\R^2)$.

We now separate $\widehat{\psi}_S$ into
\[\ba
\widehat{\psi}_S(\xi, z) &= \int_z^\infty C_1^- \exp(-\mu_1^- (z-z')) \hS (\xi, z')\dd z'\\
& + \sum\limits_{j=1,2}\int_0^z C_j^+ \exp(-\mu_j^+ (z-z')) \hS (\xi, z')\dd z' + \int_z^\infty C_2^- \exp(-\mu_2^- (z-z')) \hS  (\xi, z')\dd z'.
\ea
\]
We recall that $\beta C_j^\pm =O(1)$. Furthermore, in the regime we consider, $\Re(\mu_j^+)\gg 1$ and $-\Re(\mu_2^-)\gg 1$. Hence we deduce easily from the Lebesgue dominated convergence theorem that the three integrals involving $\mu_1^+, \mu_2^+$ and $\mu_2^-$ vanish pointwise as $\nueff\to 0$ and $\beta, |\om| \to \infty$ with $|\om| \lesssim \beta^{2/3} \nueff^{1/3}$.

There remains to consider the integral involving $\mu_1^-$. First, we note that in the scaling considered here, $|\mu_1^-| \ll |\mu|$ for $\mu \in \mathcal M\setminus \{\mu_1^-\}$, so that
\[
C_1^-\sim  \frac{1}{\nueff s^2} \prod_{\mu \in \mathcal M\setminus \{\mu_1^-\}} \frac{1}{\mu} \sim \frac{1}{\beta s},
\]
where we used the fact the product of the  three non-zero roots $M_1^+$, $M_2^+$ and $M_2^-$ of $\widetilde{Q} $ is equal to $1/s$.

It follows that
\[
\beta  s \widehat{\psi}_S(\xi, z) \sim  \int_z^\infty \exp(-\mu_1^- (z-z')) \hS  (\xi, z')\dd z',
\]
and therefore
\[
\beta (c i \xi_x- s \p_z) \widehat{\psi}_S \sim \hS  \sim \beta \widehat{\p_1 \overline{\psi}_S}.
\]
The result follows. \qed

\subsection{Additional estimates on \texorpdfstring{$G\ast S$}{the convolution of G and S}}

When we will perform the error estimates in \cref{sec:proof-thm}, we will need some further estimates on functions of the form $G\ast S$, in particular when the source term $S$ has exponential decay.
\begin{lemma}
Let $S\in L^2(\R^2, L^1(\R_+))$. We assume that there exist $C,\gamma>0$ such that for any $z>0$,
\[
\| S(\cdot, z)\|_{L^2(\R^2)}\leq C e^{-\gamma z}.
\]
We assume furthermore that $\gamma\leq \Re(\mu_j^+)/2$ for $j=1,2$.
Then
\[
\| G\ast S (z)\|_{L^2(\R^2)} \lesssim \frac{1}{\beta \gamma} e^{-\gamma z}.
\]
    \label{lem:exp-Green}
\end{lemma}
\begin{proof}
We recall that
\[
\widehat{G\ast S}(\xi, z) = \sum\limits_{j=1,2} C_j^+\int_0^z e^{-\mu_j^+ (z-z')} \hS(\xi,z')\dd z' + \sum\limits_{j=1,2} C_j^-\int_z^{\infty} e^{-\mu_j^- (z-z')} \hS(\xi,z')\dd z'.
\]
Therefore
\[
\| G\ast S (z)\|_{L^2(\R^2)} \leq C \beta^{-1} \sum\limits_{j=1,2} \int_0^z e^{-\Re(\mu_j^+)(z-z') - \gamma z'}\dd z' + \int_z^{\infty} e^{-\Re(\mu_j^-)(z-z') - \gamma z'}\dd z'.
\]
For the second integral, we simply have, recalling that $\Re(\mu_j^-)\leq 0$,
\[
\int_z^\infty e^{-\Re(\mu_j^-)(z-z') - \gamma z'}\dd z'\leq \frac{1}{\gamma - \Re(\mu_j^-)} e^{-\gamma z} \leq \frac{1}{\gamma} e^{-\gamma z}.
\]
For the first integral, we observe that the assumption on $\gamma$ ensures that 
 $\gamma\notin \{\Re(\mu_1^+), \Re(\mu_2^+)\}$. Hence the estimate is similar to the one of the first integral and we find
\[
\int_0^z e^{-\Re(\mu_j^+)(z-z') - \gamma z'}\dd z'\leq \frac{e^{-\gamma z }- e^{-\Re(\mu_j^+) z}}{\Re(\mu_j^+)-\gamma} \leq \frac{1}{\gamma} e^{-\gamma z}.
\]

\end{proof}

In the construction of an approximate solution, we will often refer to quantities of the type $G\ast S$, where $S$ is a smooth function with bounded gradient, as an ``interior term'', by opposition to ``Munk boundary layer terms'' of the type $\mathcal F^{-1} (\sum_j c_j e^{-\mu_j^+ z})$. However one should keep in mind that $G \ast S$ contains a boundary layer term of lower order, as the following result shows:
\begin{lemma}
Let $k\geq 0$, and let $S\in L^2(\R^2_{x,y}, W^{1, k+1}(\R_+))$.
Then $G\ast S\in L^2(\R^2, C^{k+1}(\R))$, and for all $z\geq 0$,
\begin{align*}
  \p_z^{k+1} (G\ast S)(z)=& G \ast \p_z^{k+1} S (z)+ \sum_{m=0}^k \p_z^{k-m}G(\cdot,z)\ast_{x,y}\p_z^m S(\cdot,z=0 )\\
  =& G \ast \p_z^{k+1} S (z)+ \sum_{m=0}^k \sum_{j=1}^2 \mathcal{F}^{-1} \left( C_j^+ (-\mu_j^+)^{k-m} \widehat{\p_z^m S}(\xi, 0) e^{-\mu_j^+ z}\right).
\end{align*}

In particular,
\[
\ba
\| \p_z^{k+1} (G\ast S)(z)\|_{L^2(\R^2)}\lesssim &\; \frac{1}{\beta} \Big( \| \p_z^{k+1} S\|_{L^2((\R^2), L^1(\R_+))} \\&\; + \sum_{m=0}^k\sum_{j=1}^2 | \mu_j^+|^{k-m}\| \p_z^m S(\cdot, 0)\|_{L^2(\R^2)}  e^{-\mu_j^+z} \Big).
\ea
\]
\label{lem:CL-Gstar}
\end{lemma}

\begin{proof}
For $k=0$, the formula follows easily from integrating by parts the right-hand side of
\[
\p_z \widehat{G\ast S}(\xi, z)=\int_0^\infty \p_z \cG(\xi, z-z') \hS(z')\dd z'.
\]
For $k\geq 1$, we argue by induction and obtain the desired result. 
\end{proof}


\section{Construction of the first order  terms of the expansion}\label{sec:correcteurs}

The purpose of this section is to determine completely and explicitely $u^0$ and $u^1$, following the computations of \cref{sec:heuristique}.
This construction will be generalized later in \cref{lem:approx-high-order} where we will construct a solution at arbitrary order.
However, since the main order terms $u^0$ and $u^1$ are the most important ones for applications, since they are used in our numerical illustrations of the separation phenomenon (see \cref{fig:illustration}), we isolate their derivation in the present Section. We emphasize that the term $u^1$ is the first that will allow us to see a clear difference with a 2D Munk model without topography, see \cref{fig:illustration} as well.

Let us recall a few features of $u^0$ and $u^1$, which we derived in \cref{sec:heuristique}. First, $u^0=(\nabla_h^\bot p^0, 0)$, where $p^0$ satisfies \eqref{systeme_interieur}. As a consequence, $p^0= p_{\mathrm{i}}^0 + p_{\mathrm{M}}^0$, where $p_{\mathrm{i}}^0 =\beta G\ast (\nabla_h^\bot \cdot f_h)$ and $ p_{\mathrm{M}}^0=\cF^{-1}(\sum\limits_{j=1,2} c_j^0 e^{-\mu_j^+ z})$ (see \cref{prop:QG3D-visc}), and $c_j^0$ are given by \eqref{def:cj0}.
Furthermore, we recall that there is no Ekman boundary layer within $u^0$ and $u^1$, so that we may also write
 $u^1=u^1_{\mathrm{i}}+u_{\mathrm{M}}^1$.

We introduce the following notations, for $v_h$ a 2D vector field and $\varphi$ a scalar function,
$$ \begin{cases}
	L^1 v_h =  (\p_t -\Delta_\nu )v_h + \beta y v_h^\perp \\
	\widetilde{L^2} \varphi \ = \nabla^\perp_h \cdot L^1 \nabla_h^\perp\varphi  = (\p_t \Delta_h + \beta \p_1 -\Delta_\nu\Delta_h)\varphi \\
	L^2 \varphi \ = (\p_t \Delta + \beta \p_1 -\Delta_\nu\Delta_h)\varphi= \widetilde{L^2} \varphi + \p_t \p_3^2 \varphi.
\end{cases} $$

We then recall the following definitions, which were given in Section \ref{sec:heuristique}, namely \eqref{u1hh} and \eqref{eqsurcorrecteurh}:
\begin{definition} The main order term $u^0$ and the first corrector $u^1$ are defined as follows:
\begin{itemize}
    \item $u^0=(\nabla_h^\bot p^0, 0)$, where $L^2 p^0= \beta \nabla_h^\bot \cdot f_h$ and $p^0=\p_n p^0=0$ on $\p\Om;$
    \item The first order pressure corrector $p^1$ satisfies
    \be\label{eqsurp1b}
L^2 p^1 = \nabla_h^\perp \cdot L^1 \bigg( \beta f_h^\perp + L^1 \nabla_h p^0 \bigg).
    \ee
More precisely, $p^1=\overline{p}^1 + \cF^{-1}(\sum_j c_j^1 e^{-\mu_j^+ z} )$, where $\overline{p}^1:=G\ast \left(\nabla_h^\perp \cdot L^1 \bigg( \beta f_h^\perp + L^1 \nabla_h p^0 \bigg)\right)$,
the roots $\mu_j^+$ are defined in \cref{lem:characteristic} and the coefficients $c_j^1$ by \eqref{coeffc1}, with 
   \be\label{def:baru1}
\overline{u}^1_h:= - \beta f_h^\perp - L^1 \nabla_h p^0  + \nabla_h^\perp \overline{p}^1,\ \overline{u}^1_3=-\p_t \p_3 p^0.
\ee
\item The first order corrector for the velocity is then defined as
    	\begin{subequations}\label{u1-with-p1}
		\begin{align}
			u^1_h  & = - \beta f_h^\perp - L^1 \nabla_h p^0  + \nabla_h^\perp p^1, \label{u_1_h}\\
			u^1_3 &= -\p_t \p_3 p^0 .\label{u_1_3}
		\end{align} 
	\end{subequations}

\end{itemize}
     \label{lemma:structure-u-1}
\end{definition}

\begin{remark}\label{rem:reg-u1} \begin{itemize}
    \item  The derivation of these equations will be performed in \cref{ssec:derivation}.   Note that $p^0$ and $p^1$ are well-defined thanks to \cref{prop:QG3D-visc}. As a consequence,  if $u^0\in H^m(\R^2_{x,y}\times \R^+_z)$, then $u^1\in H^{m-6}(\R^2_{x,y}\times \R^+_z)$.
    \item There is a crucial difference between \eqref{u_1_h} and \eqref{u_1_3}: the first one requires to know $p^1$, whereas the second merely requires to know $p^0$. 
    Note that, for the first two order terms, we would obtain the same formulas by merely considering a linearized version of the primitive equations with the same scaling (i.e. the vertical momentum equation is replaced by the hydrostatic equilibrium). 
    \item The right-hand side term in \eqref{eqsurp1b} can be understood as a commutator term. Indeed, if we had $\nabla_h^\perp \cdot L^1 =  L^1 \cdot \nabla_h^\perp$ (which is false, as we will see), then the  right-hand side would simply be $L^1 \nabla_h\cdot f_h$, which is zero under realistic modelling asumptions (the wind forcing should mainly be 2D and  divergence free).
\end{itemize}
\end{remark}

Our next result consists in algebraic manipulations on the right-hand side, which are useful to compute numerically the approximate solution:

\begin{lemma}[For numerical simulation purposes]
	Equation \eqref{eqsurp1b} can be rewritten as
	\[
L^2 p^1=\beta (\p_t - \Delta_\nu ) \nabla_h \cdot f_h + \beta^2 y \nabla_h^\bot \cdot f_h - 2 \beta y  \p_t \p_3^2 p^0 
 + \beta^2 f_1  + 2 \beta (\p_t -\Delta_\nu ) \p_2 p^0 - \nu_h \beta \p_2 \Delta_h p^0.
        \]
   
\label{lem:RHS-num}
\end{lemma}

\subsection{Formal derivation of the equation for the first order corrector}
\label{ssec:derivation}
Let us first remark that the horizontal part of the first equation of our Boussinesq-type system 
\eqref{NS-rho}
can be written in terms of $L^1$ as \be\label{l1uh} L^1 u_h + \varepsilon^{-1} u_h^\perp = - \varepsilon^{-1} \nabla_h^\perp p + \beta f_h .\ee

Moreover, taking only into account the terms larger than $O(\varepsilon^{-2})$ in the vertical part formally yields  the hydrostatic balance at  orders $0$ and $1$, namely
\be\label{d3pj}\p_3 p^j = -\rho^j, \quad j=0,1. \ee

Finally, the equation on the vertical transport can also be looked at at these two orders and entails
\be\label{rhoj}\p_t \rho^j = u^{j+1}_3, \quad j=0,1. \ee

By writing $u = u^0+\varepsilon u^1 + \varepsilon^2 u^2 + \text{l.o.t.}$, (where $\text{l.o.t.}$ stands for lower order terms), and writing the same for $p$, \eqref{l1uh} gives 
\begin{subequations}
	\begin{align}
		L^1 u^0_h + (u^1_h)^\perp + \nabla_h p^1 & = \beta  f_h  \label{u1} \\
		L^1 u^1_h + (u^2_h)^\perp + \nabla_h p^2 & = 0  \label{u2} ,
	\end{align}
\end{subequations}
that should be understood together with \eqref{d3pj} and \eqref{rhoj}.

Note that \eqref{u1} immediately gives \eqref{u_1_h}. 
Moreover, the combination of \eqref{d3pj} and \eqref{rhoj} yields 
\be\label{lienpju1}-\p_t\p_3 p^j = u^{j+1}_3,\ee
for $j=0,1$. Hence for $j=0$, we retrieve \eqref{u_1_3}. Now, to obtain the equation on $p^1$, we plug \eqref{u_1_h} into \eqref{u2}, and  take the horizontal curl. This yields
$$  \nabla_h^\perp \cdot L^1 \bigg( - \beta f_h^\perp - L^1 \nabla_h p^0 \bigg)+ \underbrace{\nabla_h^\perp \cdot L^1 \bigg(\nabla_h^\perp p^1  \bigg)}_{\widetilde{L^2} p^1} + \nabla_h \cdot u^2_h =0, $$
by definition of $\widetilde{L^2}$. Now it suffices to add $\p_3$\eqref{lienpju1} for $j=1$, to get 
$$ \nabla_h^\perp \cdot L^1 \bigg( - f_h^\perp - L^1 \nabla_h p^0 \bigg)+ \underbrace{\widetilde{L^2} p^1 +\p_t\p^2_3p^1}_{L^2 p^1} + \underbrace{\nabla_h \cdot u^2_h + \p_3 u^2_3}_{=0} =0  $$
and this gives precisely the equation on $p^1$.

For the regularity mentioned in \cref{rem:reg-u1}, replacing $\nabla_hp^0$ by $(u^0_h)^\perp$, the composition  of $\nabla_h\cdot L^1$ and $L^1$ counts for at most 5 derivatives.
Notice that we could use the regularisation of the convolution with the Green function to minimize the loss of the regularity.
However the associated estimates would depend on
the small parameters, which we would like to avoid.

\subsection{Proof of  Lemma \texorpdfstring{\ref{lem:RHS-num}}{4.4}}

To compute the right-hand side of \eqref{eqsurp1b}, we start with the term in $p^0$. First, we have
$$L^1 \nabla_h p^0  =  (\p_t -\Delta_\nu )\nabla_h p^0 + \beta y \nabla_h^\perp p^0 $$
by definition. Then, applying $L^1$ again, we find
$$\ba 
L^1 ( L^1 \nabla_h p^0 ) =& 
(\p_t -\Delta_\nu )^2 \nabla_h p^0 + (\p_t -\Delta_\nu )\bigg(\beta y \nabla_h^\perp p^0 \bigg) + \beta y (\p_t -\Delta_\nu ) \nabla_h^\perp p^0 -(\beta y)^2 \nabla_h p^0
\ea$$
and we now should compute the commutator between $\beta y$ and $(\p_t-\Delta_\nu)$, which is

$$ (\p_t -\Delta_\nu ) \bigg(\beta y \nabla_h^\perp p^0 \bigg) = \beta y (\p_t -\Delta_\nu ) \nabla_h^\perp p^0 -\nu_h \beta \p_2 \nabla_h^\perp p^0. $$

Therefore, we have obtained
$$L^1 ( L^1 \nabla_h p^0 ) =(\p_t -\Delta_\nu )^2  \nabla_h p^0 + 2 \beta y (\p_t -\Delta_\nu ) \nabla_h^\perp p^0 -\nu_h \beta \p_2 \nabla_h^\perp p^0 -(\beta y)^2 \nabla_h p^0.$$
Taking the horizontal curl yields
$$\ba 
\nabla_h^\perp \cdot L^1(L^1(\nabla_h p^0)) = 
 & 2\beta  y (\p_t -\Delta_\nu ) \Delta_h p^0 + 2 \beta (\p_t -\Delta_\nu ) \p_2 p^0 - \nu_h \beta \p_2 \Delta_h p^0 + 2\beta^2 y \p_1 p^0 \\
=& 2 \beta y \widetilde{L^2} p^0 + 2 \beta (\p_t -\Delta_\nu ) \p_2 p^0 - \nu_h \beta \p_2 \Delta_h p^0.
\ea $$

For the term in $f$, we find that 
$$\nabla_h^\perp \cdot L^1 f_h^\perp = (\p_t - \Delta_\nu ) \nabla_h \cdot f - \beta y \nabla_h^\perp\cdot f_h + \beta f_1. $$

Adding up the last two equalities and recalling that $L^2 p^0= \beta \nabla_h^\bot \cdot f_h$ leads to the desired result. \qed

We deduce from \cref{lem:RHS-num} another decomposition for $p^1$, which highlights an interesting structure which can be used, for instance, to minimise numerical errors when one performs numerical simulations.
\begin{corollary}
\label{cor:decomp-p1-bis}
Assume that $\nabla_h\cdot f_h=0$. We have
\be\label{eq:decomp-p1-bis}
p^1 = \beta \p_2^{-1} (y \p_2 p^0) + \widetilde{p}^1 + \cF^{-1} \left( \sum_{j=1,2} \tilde c_j^1 e^{-\mu_j^+ z}\right),
\ee
where $\p_2^{-1}$ is the operator whose Fourier symbol is $(i\xi_y)^{-1}$, and
\[
\widetilde{p}^1= G \ast \left(\beta^2 \p_2^{-1} \p_1 f_2 - 2 \beta y \p_t \p_3^2 p^0 -\nu_h \beta \p_2 \Delta_h p^0\right),
\]
and 
\be \label{coeffc1-tilde}
 \sum_{j\in \{1,2\}} \tilde c_j^1(\xi,\om) \begin{pmatrix}
    -i\xi_y\\
    c i \xi_x+ s\mu_j
\end{pmatrix}= - \widehat{\widetilde{ u}^1_h}\vert_{z=0},
\ee
    with
\[
\widetilde{ u}^1_h= - \beta f_h^\bot - (\p_t - \Delta_\nu) \nabla_h p^0 + \begin{pmatrix}
    0\\ - \beta \p_1 \p_2^{-1} p^0
\end{pmatrix} + \nabla_h^\bot \widetilde{p}^1.
\]
Additionally,
\[
u^1_h = \widetilde{ u}^1_h + \nabla_h^\bot \cF^{-1} \left( \sum_{j=1,2} \tilde c_j^1 e^{-\mu_j^+ z}\right).
\]

\end{corollary}

\begin{proof}
We first observe that 
\[
L^2(\beta y p^0)= \beta y L^2 p^0 + 2 \beta (\p_t - \Delta_\nu) \p_2 p^0.
\]
It follows that when $\nabla_h\cdot f_h=0$, $p^1$ satisfies
\[
L^2 p^1= L^2(\beta y p^0) + \beta^2 f_1 - 2 \beta y \p_t \p_3^2 p^0 - \nu_h \beta \p_2 \Delta_h p^0.
\]
We further write  the term $f_1$ as
\[
f_1= \p_2^{-1} \p_2 f_1 = - \p_2^{-1} \nabla_h^\bot\cdot f_h + \p_2^{-1} \p_1 f_2,
\]
so that
\[
L^2 p^1= L^2(\beta y p^0 - \beta \p_2^{-1} p^0) +   \p_2^{-1} \p_1 f_2 - 2 \beta y \p_t \p_3^2 p^0 -\nu_h \beta \p_2 \Delta_h p^0.
\]
Now, notice that $yp^0 - \p_2^{-1} p^0= \p_2^{-1} (y \p_2 p^0)$. We deduce from \cref{prop:QG3D-visc} that formula \eqref{eq:decomp-p1-bis} holds, for some coefficients $\tilde c_j^1$ which remain to be determined.

From there, we have
\[
u^1_h = - \beta f_h^\bot - (\p_t - \Delta_\nu) \nabla_h p^0 - \beta y \nabla_h^\bot p^0
+ \nabla_h^\bot \left(\beta \p_2^{-1} (y \p_2 p^0) + \widetilde{p}^1 + \cF^{-1} \left( \sum_{j=1,2} \tilde c_j^1 e^{-\mu_j^+ z}\right)\right).
    \]
We observe that $- \beta y \nabla_h^\bot p^0$ partially cancels out with the first term in $\nabla_h^\bot p^1$, and we obtain the announced formula for $u^1_h$ and $\widetilde{u}^1_h$.
The formula for the coefficients $\tilde c^1_j$ follows, recalling that $u^1\vert_{\p\Om}=0$.
\end{proof}

\subsection{Effective boundary conditions for \texorpdfstring{$p^1$}{p1} and separation mechanism}

We conclude this Section with some asymptotic formulas for the traces of $p^1$ and $\p_z p^1$ on $\p\Om:$

\begin{lemma}
    \label{lem-traces-p1}
    Assume that \eqref{hyp:parameters} is satisfied, and assume furthermore that $\nu_3\ll \nu_h^{4/3} \beta^{-1/3}$.
The following estimates hold at $z=0$:
\begin{eqnarray*}
    \p_1 p^1\vert_{z=0}&=&\beta f_1\vert_{z=0} + O (\beta^{2/3} \nueff^{1/3}) \ll \beta^{4/3} \nueff^{-1/3} ,\\
    \p_2  p^1\vert_{z=0}&=& \beta f_2\vert_{z=0} - \beta p_{\mathrm{i}}^0\vert_{z=0} + \frac{1}{s}\p_t \p_z p_{\mathrm{i}}^0\vert_{z=0}\\
   && + 3 s^2c \nu_h \mu_1^+\mu_2^+\p_x  p_{\mathrm{i}}^0\vert_{z=0} +s^2 \nu_h \p_z^2 G(\cdot,0)\ast_{x,y} (\beta \nabla_h^\bot \cdot f_h)(\cdot,0)+  o(\beta^{2/3} \nueff^{1/3}),
\end{eqnarray*}
where we have replaced $\mu_j^+$ by their leading order, which does not depend on $\xi$ (see \cref{lem:characteristic}).
\end{lemma}

Before tackling the proof of \cref{lem-traces-p1}, let us comment a little on these formulas. There are two points which are worth highlighting: the order of magnitudes of the different terms, and the mechanism behind separation.

\paragraph{Orders of magnitude.} We expect that $p^1$ is of order $\beta$ in $L^\infty$, with $c_j^1= O(\beta)$, so that  $ \p_1 p^1\vert_{z=0}$ should be of order $\beta \max(|\mu_j^+|)\sim \beta^{4/3}/\nueff^{1/3}$, and $\p_2 p^1\vert_{z=0}$ should be of order $\beta$. We refer to the definition of $p^1$ and to \cref{prop:QG3D-visc} for these heuristic bounds, which will be proved rigorously in \cref{lem:approx-high-order}.
However, we find that $ \p_1 p^1\vert_{z=0}$ is in fact of order $\beta$, hence much smaller than the expected value. This means that the normal derivative of the boundary layer part within $p^1$ actually vanishes at main order on the boundary.

As for $p^1$, we find that the first two terms $\beta f_2\vert_{z=0} - \beta p_{\mathrm{i}}^0\vert_{z=0}$ are indeed of order $\beta$ as expected. For the sake of completeness, we have also included the next order terms. The first one, namely $\p_t \p_z p_{\mathrm{i}}^0\vert_{z=0}$, is of order $\omega$. The last two, which arise from boundary layer contributions, are both of order $\nu_h (\beta/\nueff)^{2/3} = O (\beta^{2/3} \nueff^{1/3})\ll \beta$. Note that our assumption $| \omega| \lesssim \beta^{2/3} \nueff^{1/3}$ ensures that these three additional terms all have the same bound.

\paragraph{A separation mechanism.}
The pictures from \cref{fig:illustration} were obtained with a double gyre configuration. In other words, in the vicinity of $y=0$, we take 
\[
f_1(x,y,z)= \cos y \; \Phi_1(x,z),\qquad f_2(x,y,z)= \sin y \; \Phi_2(x,z),
\]
where $\Phi_1, \Phi_2$ are decaying functions. If these formulas are valid for any $y$, then it can be easily checked that $\nabla_h^\bot \cdot f_h$ is proportional to $\sin y$, and therefore $p^0_{\mathrm{i}}$ is also proportional to $\sin y$. Of course the same also holds for  the trace and normal derivative of $p^0_{\mathrm{i}}$ on the boundary. From there, recalling the definition of $c_j^0$ \eqref{def:cj0}, we infer that there exists $P^0$ such that $p^0= \sin y P^0(x,z)$.

However, the formulas from \cref{lem-traces-p1} indicate that $\p_1 p^1\vert_{z=0} $ is proportional to $\cos y$, and that $\p_2 p^1\vert_{z=0} $ is proportional to $\sin y$. Thus $p^1$ does not vanish on the line $y=0$ a priori, contrarily to $p^0$. This can also be understood from formula \eqref{eq:decomp-p1-bis}: the first two terms in the right-hand side of \eqref{eq:decomp-p1-bis} are both even with respect to $y$. To see that, notice that the multiplication by $y$, the derivative $\p_2$ and the inverse derivative $\p_2^-1 $ change the parity of the function.

As a consequence, the presence of the corrector $p^1$ within the stream function breaks the structure of the main order term: $p^0+ \eps p^1$ is no longer proportional to $\sin (y)$. This effect is precisely what is pictured in \cref{fig:illustration}, with the additional boundary layer having a visible cosine structure.

\begin{proof}
As a preliminary, let us note that assumption $\nu_3\ll \nu_h^{4/3} \beta^{-1/3}\ll \nu_h$ entails that $\nueff\sim s^2 \nu_h$.

We take the trace of \eqref{u_1_h} at $z=0$, recalling that $u^1_h\vert_{z=0}=0$ by definition of $c_j^1$. Using the properties $p^0=\p_z p^0=0$ at $z=0$, we deduce that
\[
\nabla_h p^1\vert_{z=0}= \beta f_h\vert_{z=0} + \Delta_\nu \nabla_h^\bot p^0\vert_{z=0}.
\]
The formula for $\p_1 p^1$ follows, recalling that $\mu_j^+=O(\beta^{1/3} \nueff^{-1/3})$.

As for $\p_2 p^1$, we decompose $p^0$ into $p_{\mathrm{i}}^0+ p_{\mathrm{M}}^0$. For the interior part, we use \cref{lem:CL-Gstar} {with $k=2$, and since the source term is regular enough so that $\p_z^m S \lesssim 1$}, we find that
\[
\p_1 \Delta_\nu p_{\mathrm{i}}^0\vert_{z=0}= \nueff \p_z^2 G(\cdot,0)\ast_{x,y} (\beta \nabla_h^\bot \cdot f_h)(\cdot,0) + O (\nueff (\beta/\nueff)^{1/3}).
\]
We now address the boundary layer part. We note that
\[
\p_1 \Delta_\nu e^{-\mu z}= \nu_h \p_1^3 e^{-\mu z} + O((\nu_h \mu + \nu_3 \mu^3) e^{-\mu z}),
\]
and thus, if $\nu_3\ll \nu_h^{4/3} \beta^{-1/3}$, 
\begin{align*}
 \Delta_\nu \p_1 p_{\mathrm{M}}^0\vert_{z=0}= & \cF^{-1}\left( \sum_{j=1,2} c_j^0 \nu_h (ci\xi_x + s \mu_j^+)^3\right) + o(\beta^{2/3} \nu_h^{1/3})\\
 =&\cF^{-1}\left( \sum_{j=1,2} c_j^0 \nueff s (\mu_j^+)^3\left(1 + 3 \frac{ci\xi_x}{s\mu_j^+}\right) \right)+ o(\beta^{2/3} \nu_h^{1/3}).
\end{align*}
Using the characteristic equation \eqref{eq:characteristic} together with the definition of $\mu_j^+$, we find that $i\om \mu_j^+ + \beta s = s^2 \nueff (\mu_j^+)^3 + O (\beta^{2/3} \nueff^{1/3})$. 
Recalling the definition of the coefficients $c_j^0$ \eqref{def:cj0}, we infer
\[
\ba
\cF^{-1}\left( \sum_{j=1,2} c_j^0 \nueff s (\mu_j^+)^3\right)&= -\frac{1}{s} \p_t \p_z  p_{\mathrm{M}}^0\vert_{z=0} + \beta  p_{\mathrm{M}}^0\vert_{z=0} +  O (\beta^{2/3} \nueff^{4/3})\\&=  \frac{1}{s} \p_t \p_z  p_{\mathrm{i}}^0\vert_{z=0} - \beta  p_{\mathrm{i}}^0\vert_{z=0}+  O (\beta^{2/3} \nueff^{4/3}).\ea
\]
Furthermore, inverting the system \eqref{def:cj0}, we also find that
\[
c_1^0= - \frac{\mu_2^+}{\mu_2^+-\mu_1^+} \widehat{p_{\mathrm{i}}^0}\vert_{z=0} + O(\nueff^{1/3} \beta^{-1/3}),\quad c_2^0=  \frac{\mu_1^+}{\mu_2^+-\mu_1^+} \widehat{p_{\mathrm{i}}^0 }\vert_{z=0} + O(\nueff^{1/3} \beta^{-1/3}). 
\]
We deduce that
\[
\sum_{j=1,2} c_j^0 \mu_j^2= \mu_1^+ \mu_2^+  \widehat{p_{\mathrm{i}}^0 }\vert_{z=0} + O(\beta^{1/3}/\nueff^{1/3}).
\]
The formula follows.
    \end{proof}

\section{Construction of Ekman boundary layers}
\label{sec:Ekman}

The purpose of this section is to construct the Ekman boundary layers associated with equation \eqref{NS-rho}, or rather with its $f$-plane version
\begin{align}\label{NS-rho-fplan}
\p_t u + \frac{1}{\varepsilon} e_3 \wedge u + \frac{1}{\varepsilon}\begin{pmatrix}\nabla_h p\\ \delta^{-2} \p_3 p \end{pmatrix}  - \nu_h \Delta_h u - \nu_3 \p_3^2 u=&\ \frac{1}{\eps \delta^2}\begin{pmatrix}
    0\\ 0\\ -  \rho
\end{pmatrix}\quad \text{in }\Om,   \\
\p_t \rho - \frac{1}{\varepsilon } u_3=&\ 0\quad \text{in }\Om,   \\
\Div u=&\ 0\quad \text{in }\Om.   
\end{align} 
Once again, since the above system has constant coefficients and the boundary of $\Omega$ is flat, we look for exponential solutions which decay far from the boundary (see \cite{gerard2008remarks} for a general presentation of this methodology).
Furthermore, since we have already investigated boundary layer solutions with a quasi-geostrophic structure in Section \ref{sec:QG}, we will be interested in solutions of \eqref{NS-rho-fplan} which are \textit{not} quasi-geostrophic at main order.
One final requirement will be that the $\beta$ term, which has been discarded in \eqref{NS-rho-fplan}, is indeed negligible compared to the other terms in the equation (time derivative, stratification, etc.)

Our main result on this system is the following:
\begin{proposition}\label{prop:Ekman}
Let $0<\varepsilon, \delta\ll 1$, and let $\om\in \R^*$ such that $|\omega\varepsilon| \ll 1$.    
Assume that equation \eqref{NS-rho-fplan} has a non-trivial solution of the form
    \[
\begin{pmatrix}
u\\
   p\\
   \rho
\end{pmatrix}
=\exp(i(\omega t + \xi_x x + \xi_y y) - \lambda z) \mathsf{U},
        \]
        with $\mathsf{U}\in \bC^5\setminus\{0\}$, $(\xi_x, \xi_y)\in \R^2$ with $|\xi| \ll |\om/\nueff|^{1/2}$, and $\Re(\lambda)>0$, $\Re(\lambda)\gg 1$. 
        Then 
        \[
\text{either }\lambda\sim e^{-i\frac{\pi }{4} \sgn(\om)}\dfrac{|\tan \alpha| }{\nueff^\frac  1 2 \varepsilon \vert \om \vert^\frac 1 2  }\quad \text{or}\quad \lambda\sim e^{i\frac{\pi }{4}\sgn(\om)} \dfrac{\vert \om \vert^\frac 1 2 }{\nueff^\frac  1 2  \vert \sin \alpha \vert},
                \]
where $\nueff=\nu_h \sin^2 \alpha + \nu_3 \cos^2 \alpha$.

\begin{enumerate}
    \item If $\lambda \sim e^{-i\frac{\pi }{4} \sgn(\om)}\dfrac{|\tan \alpha| }{\nueff^\frac  1 2 \varepsilon \vert \om \vert^\frac 1 2  }$, then $\widetilde{ \mathsf{U}}\simeq e_x$\footnote{$\widetilde{ \mathsf{U}}$ is the 3D vector consisting of the first three components of $\mathsf U$.}. In this case, the hydrostatic balance is satisfied at main order, and the viscous dissipation is balanced by rotation in the $e_y$ direction.

    \item If $\lambda \sim  e^{i\frac{\pi }{4}\sgn(\om)} \dfrac{\vert \om \vert^\frac 1 2 }{\nueff^\frac  1 2  \vert \sin \alpha \vert}$, then $\widetilde{\mathsf{U}}\simeq e_y$. In this case, the hydrostatic balance and the quasi-geostrophic balance are satisfied at main order.
   
\end{enumerate}

\end{proposition}
\begin{remark}
    As announced in the introduction, within this section we do not assume that \eqref{hyp:parameters} is satisfied, but we keep general values for the coefficients. Remarkably, the expressions of $\lambda_1$ and $\lambda_2$ at main order do not depend on the aspect ratio $\delta$.
\end{remark}

\begin{remark}[Root $\lambda_2$ will be discarded]
    In the construction of the approximate solution, we will always discard the second Ekman boundary layer $\lambda_2$. Indeed, in the regime we consider (see assumption \eqref{hyp:parameters}), the $\beta$ effect, which we have neglected, is in fact dominant in the boundary layer associated with $\lambda_2$, and therefore the derivation of the expression of $\lambda_2$ is not valid.

    In fact, looking more closely at the expression of $\lambda$ and at the associated eigenvector, we will see that 
    $\lambda_2$ coincides with the root $\mu_1^+$ within Munk boundary layers in the regime described in \cref{rem:regime}. In this case, this Ekman boundary layer is redundant with the Munk layer, and should be discarded in order to avoid any artifical undeterminacy of the coefficients.  

    This phenomenon (quasi-geostrophy of one of the Ekman layers, potential redundancy with one of the Munk layers) was completely absent from previous studies of Ekman layers, and is strongly associated with the fact that the boundary is not horizontal, hence Munk and Ekman boundary layers are localized within the same region. 
    \label{rem:redundancy-Ekman}
\end{remark}

\begin{remark}[Balance within the first Ekman layer]
As mentioned in the Introduction, viscous dissipation is balanced by rotation in the $e_y$ direction in the first Ekman layer, as in classical Ekman layers. However, in the $e_x$ direction, the balance is more complicated: viscous dissipation is balanced by a combination of pressure and stratification, and in particular to the non-hydrostatic part of the pressure. Rotation is negligible in this direction.
\end{remark}

\subsection{Identifying the characteristic roots}
The first step is to compute the characteristic equation satisfied by $\lambda$. To that end, we choose to work in the local variables $(x,y,z)$, using the changes of coordinates
\be \label{change_coos}
\ba
e_x = ce_1 + se_3,\qquad &e_z = ce_3 - se_1 \\ e_1 = ce_x -se_z,\qquad &e_3 = ce_z +se_1,
\ea
\ee 
where we have used the shortcut $s=\sin \alpha$, $c=\cos \alpha$. 
For further purposes, let us also introduce the short-hand notation
 $$
    \tilde s  =  s + i \xi_x c \lambda^{-1}, \quad \tilde c= c - i \xi_x s  \lambda^{-1},
$$ 
so that $\tilde s \simeq s, \tilde c \simeq c$ for $|\lambda|\gg |\xi_x|$. 
We also note that $s \tilde s + c \tilde c= 1$,  $\tilde s c - \tilde c s= i \xi_x \lambda^{-1}$, and
\[
\p_1e^{i(\xi_x x + \xi_y y) - \lambda z}= \lambda \tilde s e^{i(\xi_x x + \xi_y y) - \lambda z},\quad \p_3e^{i(\xi_x x + \xi_y y) - \lambda z}=- \lambda \tilde c e^{i(\xi_x x + \xi_y y) - \lambda z}.
\]

Setting 
\[
\mathsf{U}'=\begin{pmatrix}
c \mathsf{U}_1 + s   \mathsf{U}_3\\ \mathsf{U}_2\\   -s \mathsf{U}_1 +   c \mathsf{U}_3\\ \mathsf{P}\\ \mathsf{R}
\end{pmatrix}:= 
\begin{pmatrix}
   \mathsf{U}'_x\\\mathsf{U}'_y\\ \mathsf{U}'_z \\ \mathsf{P}\\ \mathsf{R}
\end{pmatrix},
    \]
we infer that $\mathsf{U}'$ satisfies the linear system $\mathcal A \mathsf{U}'=0$, where
\be \label{ekman_lin}
\mathcal A :=\begin{pmatrix}
 r & - \frac{c}{\varepsilon} & 0 & \lambda \tilde s\frac{c}{\varepsilon}- \lambda \tilde c \frac{s}{\varepsilon \delta^2} & \frac{s}{\varepsilon \delta^2} \\
  \frac{c}{\varepsilon} & r & - \frac{s}{\varepsilon} & {\frac{i\xi_y}{\varepsilon }} & 0\\
  0 & \frac{s}{\varepsilon} & r &  -\lambda \tilde s \frac{s}{\varepsilon} - \lambda \tilde c \frac{c}{\varepsilon \delta^2}& \frac{c}{\varepsilon \delta^2} \\
  i \xi_x & i \xi_y &-\lambda & 0 &0\\
  - \frac{s}{\varepsilon} & 0 & - \frac{c}{\varepsilon} & 0 & i\om 
\end{pmatrix}
\ee
in which we set, keeping the short-hand $\nueff= \nu_h s^2 + \nu_3 c^2$,
\begin{align*}
r:=&  \, i\om -\nu_h ( i c\xi_x + \lambda s)^2 + \nu_h \xi_y^2 - \nu_3 (is\xi_x - c \lambda )^2\\
=& \,  i\om - \nueff \lambda^2 + 2 i c  s  \xi_x (\nu_3 - \nu_h )\lambda+ \nu_h (c^2 \xi_x^2 + \xi_y^2) + \nu_3 s^2 \xi_x^2 .
\end{align*}

Note that the fourth line of $\mathcal A$ corresponds to the divergence free condition, and the fifth one to the conservation of mass.

System \eqref{ekman_lin} has a non-trivial solution if and only if $\det \mathcal A=0$, and it can be checked that $\det \mathcal A$ is a polynomial of degree  6 in $\lambda$. In order to simplify its computation, let us reduce the system to a $2\times 2$ system on $(\mathsf{U}'_x, \mathsf{U}'_y)$.
Using the divergence-free condition and mass conservation, we have
\[
\ba 
\mathsf{U}'_z = &\lambda^{-1} (i\xi_x \mathsf{U}'_x + i \xi_y \mathsf{U}'_y)\\
\mathsf{R} =& \frac{1}{i\om \varepsilon} (s \mathsf{U}'_x  + c \mathsf{U}'_z )
= \frac{1}{i\om \varepsilon}\left( \tilde s \, \mathsf{U}'_x  + \frac{ i \xi_y c}{\lambda}  \mathsf{U}'_y )\right).
\ea
\]
Eventually, we project the momentum equation onto $e_3$, which amounts to multiplying the first line of $\mathcal A$ by $s$ and the third line by $c$. We obtain
\[
r\left( s \mathsf{U}'_x +  c \mathsf{U}'_z\right) - \frac{\lambda \tilde c}{\varepsilon \delta^2} \mathsf{P} + \frac{1}{\varepsilon \delta^2 }\mathsf{R}=0.
\]
Replacing $\mathsf{U}'_z, \mathsf{R}$ by their previous expressions leads to
\[
\mathsf{P}=\frac{\varepsilon \delta^2}{ \lambda \tilde c }\left( r + \frac{1}{i\om \varepsilon^2 \delta^2}\right)\left[ \tilde s  \mathsf{U}'_x + i \xi_y \lambda^{-1} c \mathsf{U}'_y\right],
\]
and the combination of the pressure gradient and stratification in the first line of \eqref{ekman_lin} is
\begin{align}\label{pression-Ekman-non-hydro}
  & \left[  \lambda \tilde s \frac{c}{\varepsilon}  - \lambda \tilde c  \frac{s}{\varepsilon \delta^2}\right] \mathsf{P} + \frac{s}{\varepsilon \delta^2}\mathsf{R}\\
=&\left(\tilde s \mathsf{U}'_x + i c \xi_y \lambda^{-1} \mathsf{U}'_y\right) \left[-s r + \frac{\delta^2 c \tilde s}{\tilde c} \left( r + \frac{1}{i\om \varepsilon^2 \delta^2}\right)\right].
\nonumber
\end{align}
Plugging these expressions into the system for $(\mathsf{U}'_x, \mathsf{U}'_y)$, we find eventually
\be\label{reduced-Ekman-lin}
\begin{pmatrix}
r (1- s  \tilde s) + \tilde s^2 \dfrac{c\delta^2}{\tilde c} \left( r + \dfrac{1}{i\om \varepsilon^2 \delta^2}\right) 
& - \dfrac{c}{\varepsilon} + i c \xi_y \lambda^{-1} \left[-s r + \dfrac{\delta^2 c \tilde s}{\tilde c} \left( r + \dfrac{1}{i\om \varepsilon^2 \delta^2}\right)\right]
\\~&~\\
\dfrac{\tilde c }{\varepsilon} + \dfrac{i\xi_y \delta^2\tilde s}{ \lambda \tilde c}\left( r + \dfrac{1}{i\om \varepsilon^2 \delta^2}\right)& r - \dfrac{si \xi_y}{\varepsilon\lambda} - \dfrac{\delta^2\xi_y^2 c}{ \lambda^2\tilde c}\left( r + \dfrac{1}{i\om \varepsilon^2 \delta^2}\right) 
    \end{pmatrix}
    \begin{pmatrix}
\mathsf{U}'_x\\ \mathsf{U}'_y
    \end{pmatrix}=0.
\ee
Computing the determinant of the $2\times 2$ matrix in the left-hand side, we obtain after some simplification the characteristic equation
\be\label{determinant} r^2 + \dfrac{1}{\varepsilon^2} + \delta^2 r \bigg(r + \dfrac{1}{i \om \varepsilon^2 \delta^2}\bigg)\bigg(\dfrac{\tilde s^2}{\tilde c^2} - \dfrac{\xi_y^2}{\lambda^2 \tilde c^2}\bigg) = 0.\ee
It follows from the expressions of $\tilde c, \tilde s$ that (after multiplication by $\lambda^2\tilde c^2$) the left-hand side is indeed a polynomial of degree 6 in $\lambda$. We emphasize that this expression is exact: at this stage, no approximation has been made.

We now look for roots $\lambda$ such that $|\lambda|\gg 1$ and $| \lambda| \gg | \xi|$. For such roots, we may approximate $\tilde s$, $\tilde c$ by $s$ and $c$ respectively.
Expanding the last term in \eqref{determinant} and assuming that $\delta\ll 1$,
 the equation for $r$ becomes
\be\label{determinant-approche}
r^2 + \frac{s^2}{i\om \varepsilon^2 c^2} r + \frac{1}{\varepsilon^2}=0.
\ee
Note that this equation no longer depends on the value of the aspect ratio.
The discriminant of the equation is
\[
\Delta= - \frac{s^4}{\om^2 \eps^4 c^4} - \frac{4}{\varepsilon^2}.
\]
Since we have assumed that $|\om\varepsilon| \ll 1$, we find that the second term is negligible, and therefore the roots of \eqref{determinant-approche} are given by
\[
r_1\simeq \frac{i  s^2}{\om \varepsilon^2  c^2},\qquad r_2 \simeq  \frac{-i \om  c^2}{ s^2}.
\]
For each value of $r=r_i$, we find two distinct roots $\pm \lambda_i$ by recalling that $r=  i \om - \nueff \lambda^2 + O(\lambda \nu)$. 
Keeping only the roots with positive real parts and noticing that $|r_1| \gg |\om|$, we obtain eventually
\[
\lambda_1 \sim e^{-i\frac{\pi }{4} \sgn(\om)}\dfrac{|\tan \alpha| }{\nueff^\frac  1 2 \varepsilon \vert \om \vert^\frac 1 2  },
\qquad \lambda_2\sim e^{i\frac{\pi }{4}\sgn(\om)} \dfrac{\vert \om \vert^\frac 1 2 }{\nueff^\frac  1 2  \vert \sin \alpha \vert}.
\]

\begin{remark}
It can be checked that the two remaining complex roots of \eqref{determinant} satisfy $\lambda\sim \pm (|\xi_x|^2 + |\xi_y|^2)^{1/2}$, so that $\dfrac{\tilde s^2}{\tilde c^2} - \dfrac{\xi_y^2}{\lambda^2 \tilde c^2}\simeq -1$ and $r\simeq i \om$.  Therefore they do not correspond to boundary layer modes. 
\end{remark}

\subsection{Analysis of the generalized eigencouple associated with \texorpdfstring{$\lambda_1$}{lambda 1}}

We now analyse the behaviour of the eigenvector $\mathsf{U}^1$ associated with $\lambda_1$. Looking at the second line of the matrix in the left-hand side of \eqref{reduced-Ekman-lin}, we find that
\begin{align*}
\left| \dfrac{i\xi_y\delta^2\tilde s}{ \lambda \tilde c}\left( r + \dfrac{1}{i\om \varepsilon^2 \delta^2}\right)\right|& \lesssim |\xi_y| \frac{1}{|\om | \varepsilon^2 |\lambda_1|}\lesssim |\xi_y| \frac{\nueff^{1/2}}{\eps | \om|^{1/2}}\ll \varepsilon^{-1},\\
    \left| \dfrac{si \xi_y}{\varepsilon\lambda_1} \right| &\ll \frac{1}{\varepsilon}\ll |r_1|,\\
      \left| \dfrac{\delta^2\xi_y^2 c}{ \lambda_1^2\tilde c}\left( r + \dfrac{1}{i\om \varepsilon^2 \delta^2} \right)\right|& \lesssim |\xi_y|^2 \nueff \ll |r_1|.
\end{align*}
Hence the second line of \eqref{reduced-Ekman-lin} becomes, at main order,
\be\label{balance-lambda1}
\frac{c}{\eps} \mathsf{U}'_x + r_1 \mathsf{U}'_y\simeq 0.
\ee
Normalizing the eigenvector by choosing $\mathsf{U}'_x=1$, we obtain 
\[
\mathsf{U}'_y\simeq i \frac{c^3}{s^2}\om \eps \ll 1.
\]
From there, we infer that
\begin{align*}
    \mathsf{U}'_z&\simeq i \xi_x \lambda_1^{-1} = i \xi_x e^{i\frac{\pi}{4}\sgn(\om)} \nueff^{1/2} \eps | \om|^{1/2} | \tan \alpha|^{-1},\\
    \mathsf{P}&\simeq e^{-i\frac{\pi}{4}\sgn(\om)} \left(\frac{\nueff}{|\om|}\right)^{1/2} \sgn \alpha,\\
    \mathsf{R}&\simeq\frac{s}{i\om \varepsilon}. 
\end{align*}

\begin{remark}\label{rem:structure-ekman-valide}

The explicit expression of the eigenvector is quite informative, and shows that the structure of this Ekman boundary layer differs from the classical case without stratification. Indeed, at main order, we have $\mathsf{U_1}\simeq \cos \alpha$ and $\mathsf{U}_3 \simeq \sin \alpha$. Therefore the vertical component of the eigenvector is non zero at the main order, so that the vertical component of the interior solution  at order $\varepsilon^k$ on the boundary is balanced exactly by an Ekman layer of order $\varepsilon^k$, not $\varepsilon^{k-1}$ as it can be the case in flat settings. This is why the notion of Ekman pumping was referred to as no more relevant in the introduction. 
Additionally, it can be checked that
\[
\mathsf{P} \sim \frac{1}{\lambda_1^{-1} c} \mathsf{R},
\]
which means that hydrostatic balance is satisfied at main order within the boundary layer. 
However, computing the combination of the pressure gradient and of the stratification term in the $e_x$ component from \eqref{pression-Ekman-non-hydro}, we see that
\[
 \left[ \lambda_1 \tilde s \frac{c}{\varepsilon}  - \lambda_1 \tilde c  \frac{s}{\varepsilon \delta^2}\right] \mathsf{P} + \frac{s}{\varepsilon \delta^2}\mathsf{R}\simeq - r_1 \mathsf{U}'_x.
\]
This term exactly balances the diagonal term stemming from viscous dissipation in the $e_x$ component. In this component, the rotation term does not play a role at main order. In the second component however, the balance \eqref{balance-lambda1} is satisfied
at main order, which means that the viscous dissipation of the second component balances the Coriolis force.
It can be checked that the $\beta$ effect is indeed negligible within this boundary layer.

\end{remark}

In conclusion, we find that the boundary layer associated with the root $\lambda_1$ satisfies the following features.
\begin{itemize}
    \item Hydrostatic equilibrium is satisfied at main order. 

    \item The viscous dissipation balances the Coriolis term in the $e_y$ direction, and the combination of pressure and stratification in the $e_x$ direction, i.e. in the tangential direction that has a non-zero component along the vertical.
\end{itemize}

\subsection{Analysis of the generalized eigencouple associated with \texorpdfstring{$\lambda_2$}{lambda 2}}

 We now look at the first line of the matrix in the left-hand side of \eqref{reduced-Ekman-lin} when $r=r_2$, $\lambda=\lambda_2$. Note that
 \begin{align*}
r (1- s  \tilde s) + \tilde s^2 \dfrac{c\delta^2}{\tilde c} \left( r + \dfrac{1}{i\om \varepsilon^2 \delta^2}\right)  &=  \frac{\tilde s^2 c}{i \tilde c \om \eps^2} + O(\om) ,\\
 - \dfrac{c}{\varepsilon} + i\dfrac{\xi_y c}{\lambda} \left(\delta^2 \tilde s \dfrac{c}{\tilde c} \left( r + \dfrac{1}{i\om \varepsilon^2\delta^2}\right) - r s\right)&=  - \dfrac{c}{\varepsilon}  + \frac{\xi_y \tilde s c^2}{\om \eps^2\lambda_2 \tilde c}+ O( \om \lambda_2^{-1}).
 \end{align*}
Taking $\mathsf{U}'_y=1$, we find
\begin{align*}
\mathsf{U}'_x=& \left(  \dfrac{c}{\varepsilon}  - \frac{\xi_y \tilde s c^2}{\om \eps^2\lambda_2 \tilde c}+ O( \om \lambda_2^{-1})\right) \frac{i\tilde c \om \eps^2}{\tilde s^2 c} \left( 1 + O(\om^2 \varepsilon^2)\right)\\
=& \frac{i \om \eps \tilde c}{\tilde s^2} - i\xi_y \lambda_2^{-1} \frac{c}{\tilde s} + O(\omega^2 \varepsilon^2\lambda_2^{-1}  + (\omega\varepsilon)^3),\\
\mathsf{U}'_z=& \lambda_2^{-1} \left( i \xi_y \frac{s}{\tilde s}+  i \xi_x \frac{i \om \eps \tilde c}{\tilde s^2}\right) +   O(\omega^2 \varepsilon^2\lambda_2^{-2}  + (\omega\varepsilon)^3\lambda_2^{-1}).
    \end{align*}
From there, we get
\begin{align*}
   \mathsf{U}_1&= \frac{i\om \varepsilon \tilde c^2}{\tilde s^2} - \frac{i\xi_y \lambda_2^{-1}}{\tilde s}+   O(\omega^2 \varepsilon^2\lambda_2^{-1}  + (\omega\varepsilon)^3),\\
   \mathsf{U}_3&=\frac{i\om \varepsilon \tilde c}{\tilde s} +   O(\omega^2 \varepsilon^2\lambda_2^{-2}  + (\omega\varepsilon)^3\lambda_2^{-1}),\\
   \mathsf{P}&= \frac{1}{\lambda_2\tilde s} (1+ O(\om \varepsilon \lambda_2^{-1}  + \om^2 \varepsilon^2 )),\\
   \mathsf{R}&=\frac{\tilde c}{\tilde s}+   O(\omega \varepsilon\lambda_2^{-2}  + (\omega\varepsilon)^2\lambda_2^{-1}).
\end{align*}
In particular, it follows that
\[
 \mathcal F(\nabla_h^{\bot} {P})e^{\lambda_2 z}= \begin{pmatrix}
     -i\xi_y\\ \lambda_2 \tilde s
 \end{pmatrix}\mathsf{P}
 = \mathsf{U}_h + \begin{pmatrix}
i\om \varepsilon \frac{\tilde c^2}{\tilde s^2} +    O(\omega \varepsilon\lambda_2^{-2}  + (\omega\varepsilon)^2\lambda_2^{-1} )\\
 O(\om \varepsilon \lambda_2^{-1}  + \om^2 \varepsilon^2 )
 \end{pmatrix}.
\]
Hence the motion is \textit{quasi-geostrophic at main order}: rotation is balanced by the pressure gradient. Furthermore, writting an asymptotic expansion for $\mathsf{\tilde U}$, $\mathsf{P}$
\[
\widetilde{\mathsf{U}}^0=\begin{pmatrix}
- \frac{i\xi_y \lambda_2^{-1}}{\tilde s}\\ 1 \\0
\end{pmatrix},\quad \widetilde{\mathsf{U}}^1=\begin{pmatrix}
\frac{i\om  \tilde c^2}{\tilde s^2}\\ 0 \\\frac{i\om  \tilde c}{\tilde s}
\end{pmatrix},\quad \mathsf{P}^0= \frac{1}{\lambda_2\tilde s},
\]
we observe that, denoting $U^0,U^1,P^0 = \mathcal{F}^{-1}(\widetilde{\mathsf{U}}^0,\widetilde{\mathsf{U}}^1,{\mathsf{P}}^0)$
\[\ba
&\cF\left[(\p_t - \Delta_\nu)({U}^0_h + \varepsilon U^1_h) + \varepsilon^{-1} (U^0_h + \varepsilon U^1_h)^\bot + \varepsilon^{-1}\nabla_h P^0\right]e^{\lambda_2 z}\\
=\;& r_2 \mathsf U^0_h + (\mathsf{U}^1_h)^\bot + O(\om^2 \varepsilon)= O (\om \lambda_2^{-1} + \om^2 \varepsilon).\ea
\]
Note that this corresponds exactly to the balance described in Section \ref{sec:QG}, without the $\beta$ term. There are two possible situations:
\begin{itemize}
    \item either the $\beta$ term is negligible in the quasi-geostrophic balance for Munk layers, i.e. in equation \eqref{eq:characteristic}.
    This corresponds to the regime $\beta^{2/3} \nueff^{1/3} \ll | \om| \ll \beta^{3/4} \nueff^{1/4}$ described in \cref{rem:regime}. In this case, we observe that the root $\lambda_2$ found above is equivalent at main order to $\mu_1^+$: the second Ekman layer coincides with one of the Munk layers.

    Therefore the second Ekman layer, corresponding to $\lambda=\lambda_2$ is already included within the quasi-geostrophic part of the solution and can be discarded.

    \item or the   $\beta$ term is not negligible in the quasi-geostrophic balance for Munk layers. This corresponds to the regime $|\om| \lesssim \beta^{2/3} \nueff^{1/3}$ described in \cref{lem:characteristic}, i.e. $| \om| \lesssim \beta |\lambda_2|^{-1}$. Hence $\beta |U^0_1| \gtrsim |U^1_1|$: therefore the $\beta$ term, which has been discarded from the computation of Ekman layers, is in fact of higher order than the rotation term associated to $U^1$. In this case the computation of the Ekman layer is not valid, and the second Ekman layer must be discarded again.
\end{itemize}
In conclusion, we will never keep the eigenvector associated with $\lambda_2$.
Hence there will only be one degree of freedom within the Ekman boundary layer, namely the one associated with $\lambda_1$.

\section{Construction and proof of validity of approximate solutions}
\label{sec:proof-thm}
This section is devoted to the proof of our main results, namely \cref{prop:sol-high-order-simplified}, \cref{thm:stab-periodic} and \cref{thm:stab-Cauchy}.

We first construct an approximate solution at any order for source terms whose Fourier support lies in a compact set (\cref{lem:approx-high-order}).
We then explain how to truncate general source terms in order to ensure that this assumption is satisfied \cref{lem:truncation-source}.
Eventually we prove \cref{thm:stab-periodic} and \cref{thm:stab-Cauchy} thanks to a simple energy estimate.

\subsection{Approximate solution at any order}

The purpose of this Section is to prove the following Lemma, which is a more precise version of \cref{prop:sol-high-order-simplified} under additional assumptions on the source term.
\begin{lemma}
Let $\varepsilon>0$, and assume that $\om,\nu_h,\nu_3, \beta$ satisfy assumption \eqref{hyp:parameters}.

Let $N,m>0$ be arbitrary. Assume that $f$ satisfies (H1)-(H4) for some $q,Q$ depending on $N$ and $m$, and sufficiently large. Assume furthermore that $\hat f(\xi,z)=0$ for $|\xi| \geq R$, with $|R| \ll (\beta/\nueff)^{1/3}$, and let $\mu= \min (\Re(\mu_1^+), \Re(\mu_2^+))/2$, where $\mu_i^\pm$ are defined in \cref{lem:characteristic}.
 
Then there exists $K>0$ depending  on $N$ and $m$ and on the parameters $a,b,d,e$, and an approximate solution of the form
\[
\uapp=\sum_{k=0}^K \varepsilon^k u^k 
\]
where each term $u^k$ in the above sum can be decomposed as an interior part, a Munk boundary layer part, and an Ekman boundary layer part, namely $
u^k= u_{\mathrm{i}}^k + u_{\mathrm M}^k + u_{\mathrm E}^k$, with:
\begin{itemize}
    \item Interior part: for $j=1,3$,
    \[
\|  u_{\mathrm{i},j}^k\|_{H^m} \lesssim \beta^k \mu^{m-3/2} ,\quad \| u_{\mathrm{i},2}^k\|_{H^m} \lesssim \beta^k  \mu^{m-1/2};
\]
    \item Munk boundary layer part: for $j=1,3$,
    \[
\|  u_{\mathrm{M},j}^k\|_{H^m} \lesssim \beta^k \mu^{m-1/2} ,\quad \| u_{\mathrm{M},2}^k\|_{H^m} \lesssim \beta^k  \mu^{m-3/2};
\]
      and the term  $u_{\mathrm{M}}^k$ is exponentially small outside a boundary layer of size $\mu^{-1}$;

    \item Ekman boundary layer part: $\|  u_{\mathrm{E}}^k\|_{H^m} \lesssim \beta^k | \lambda_1|^{m-1/2} $ ; and the term  $u_{\mathrm{E}}^k$ is exponentially small outside a boundary layer of size $\Re(\lambda_1)^{-1}$.
\end{itemize}
Furthermore, $\uapp\vert_{\p\Om}=0$ and $\uapp$ satisfies \eqref{NS-rho} up to a remainder $g_{\mathrm{rem}}$ such that
\[
\sup_{t\geq 0} \int_0^\infty(1+ z^2) \| g_{\mathrm{rem}}(t,\cdot, z)\|_{L^2_{x,y}}^2 \dd z \leq \varepsilon^{N}.
\]

    \label{lem:approx-high-order}
\end{lemma}

\begin{remark}
\begin{itemize}
\item Note that each term in the asymptotic expansion will in fact depend on the parameters $\om, \beta, \nu_h$ and $\nu_3$, and therefore on $\varepsilon$. In view of the properties satisfied by $u^k$, our asymptotic expansion is in fact an expansion in powers of $\varepsilon\beta \ll 1$. However, keeping the description above allows us to perform more compact computations, to preserve the structure of the hierarchy,
and to avoid discussions on the relative sizes of the parameters. We emphasize that each term $u^k$ can be determined explicitly (and shall be to compute the next order term).

\item As already mentioned, the term ``interior'' is slightly misleading since each term $u_{\mathrm{i}}^k$ will contain a boundary layer type term (see \cref{lem:CL-Gstar}). However this boundary layer term will be weaker than the one contained in $u_{\mathrm{M}}^k$, which explains why the upper-bounds for  $u_{\mathrm{i}}^k$ are smaller than the ones for $u_{\mathrm{M}}^k$ by a power of $\mu$.

\item Assumption \eqref{hyp:parameters} ensures that $\mu$ is bounded from above and below by $C (\beta/\nueff)^{1/3}$ (see \cref{lem:characteristic}). This fact will be used repeatedly in the proof.
\end{itemize}  
\end{remark}
\begin{proof}
First, we take a similar Ansatz for the pressure and density, namely $p=\sum_k \eps^k p^k$, $\rho=\sum_k \rho^k$, where each $p^k, \rho^k $ is decomposed into an interior part, a Munk part and an Ekman part. We will construct each family recursively. We recall that the terms corresponding to $k=0$ and $k=1$ have already been constructed in \cref{sec:correcteurs}. We will in fact propagate estimates that are more precise than the ones announced in the Lemma. More precisely, we will prove that for all $k$, there exists $m_k, C_k$ such that for $m_x+ m_y + m_z\leq m_k$, for $j=1,3$,
\begin{align}
\label{est:ui13}\|\p_x^{m_x} \p_y^{m_y} \p_z^{m_z}  u_{\mathrm{i}, j}^k(z)\|_{L^2_{xy}}&\leq C_k \beta^k (e^{-\gamma z} + \mu^{(m_z-1)_+} e^{-\mu z}) ,\\ 
\label{est:ui2}\|\p_x^{m_x} \p_y^{m_y} \p_z^{m_z}  u_{\mathrm{i}, 2}^k(z)\|_{L^2_{xy}}&\leq C_k \beta^k (e^{-\gamma z} + \mu^{m_z} e^{-\mu z}) ,\\
\label{est:uM13}\|\p_x^{m_x} \p_y^{m_y} \p_z^{m_z}  u_{\mathrm{M}, j}^k(z)\|_{L^2_{xy}}&\leq C_k \beta^k \mu^{m_z} e^{-\mu z},\\  \|\p_x^{m_x} \p_y^{m_y} \label{est:uM2}\p_z^{m_z}  u_{\mathrm{M}, 2}^k(z)\|_{L^2_{xy}}&\leq C_k \beta^k\mu^{m_z+1} e^{-\mu z}.
\end{align}
Additionally, the support in the tangential Fourier variables of each term in the expansion will be included in $B(0, R)$.
The sequence $(m_k)$ will satisfy $m_{k+1}=m_k-6$ due to some finite loss of derivatives at every step. Eventually, we will choose the source term with sufficient regularity, so that $m_k\geq m$ for all $k\in \{0,\cdots, K\}$.

\noindent\textit{\textbf{Iterative construction of the interior part.}} We first consider the interior part, and we omit momentarily the subscript $\mathrm{i}$ in order not to burden the notation.
Writing the balance for the terms of order $\eps^k$ in \eqref{L1NSrho} (arguing for now as if $\om$, $\beta$, $\nu$ were of order one), we find, for $k\geq 0$,
\begin{align}
    \label{eq:uk}(\p_t - \Delta_\nu) u_h^k + \beta y (u_h^k)^{\bot} + (u_h^{k+1})^\bot + \nabla_h p^{k+1} &=\delta_{k,0} f,\\
   \label{eq:u3k} (\p_t-\Delta_\nu) u_3^k + \p_3 p^{k+3} &= - \rho^{k+3},\\
    \p_t \rho^k - u_3^{k+1}&=0.
\end{align}
It follows in particular, setting $u^k, \rho^k, p^k=0$ for $k<0$, that
\begin{equation}\label{eq:uk-recurrence}
\ba
   u^{k+1}=& \begin{pmatrix}
\nabla_h^{\bot} p^{k+1}\\ 0
    \end{pmatrix} + \begin{pmatrix}
        (\p_t - \Delta_\nu) (u_h^k)^\bot - \beta y u^k_h\\ \p_t \rho^k
    \end{pmatrix},\\
    \rho^{k+1} =& - \p_3 p^{k+1} - (\p_t - \Delta_\nu) u_3^{k-2}. \ea
\end{equation}
Hence we find that each term in the expansion is the sum of a geostrophic part $\begin{pmatrix}
    \nabla_h^\bot p^{k+1}\\0
\end{pmatrix}$, which remains to be determined, and of an explicit part, which is non geostrophic but completely determined by lower order terms.
In order to to determine $p^{k+1}$, we follow computations similar to the ones performed in \cref{sec:correcteurs} (see in particular \cref{lemma:structure-u-1}), and we find, for all $k\geq 0$,
\be
\label{eq:QG-k}
\p_t \Delta p^{k+1} + \beta \p_1 p^{k+1} - \Delta_\nu \Delta_h p^{k+1}= F^{k+1},
\ee
where the source term $F^{k+1}$ depends only on lower order terms. More precisely,
\begin{align*}
    F^{k+1}:=&-(\p_t - \Delta_\nu)^2  \divh u_h^{k} - (\p_t - \Delta_\nu)\left[  - \beta y \nabla_h^\bot \cdot u^{k}_h + \beta u^{k}_1\right]\\
    &+ \beta y \left[ (\p_t - \Delta_\nu) \nabla_h^\bot u_h^{k} + \beta y \divh u^{k}_h + \beta u^{k}_2\right]\\
    &-\beta (\p_t - \Delta_\nu) u_1^{k} + \beta^2 y u_2^{k}\\
    &+ \p_t \p_3 (\p_t - \Delta_\nu) u_3^{k-2} .
\end{align*}
It follows that $F^{k+1}$ is fully determined by $u^j$ for $j\leq k$. Using \cref{prop:QG3D-visc} and re-introducing the subscript $\mathrm{i}$, we take $p^{k+1}_\mathrm{i} =G\ast F^{k+1}_\mathrm{i}$, which completes the definition  of $(u^{k+1}_{\mathrm{i}}, \rho^{k+1}_{\mathrm{i}}, p^{k+1}_{\mathrm{i}})$.

Assume that estimates \eqref{est:ui13}, \eqref{est:ui2} are satisfied up to rank $k-1$. Then, using \cref{lem:exp-Green} and \cref{lem:CL-Gstar}, we infer that for $m_x+ m_y + m_z \leq m_{k-1}-5$, 
\[
\| \p_x^{m_x} \p_y^{m_y} \p_z^{m_z} p^k_{\mathrm{i}}(z)\|_{L^2_{x,y}}\lesssim \beta^k  (e^{-\gamma z} + \mu^{(m_z-1)_+} e^{-\mu z}).
\]
We infer that estimates \eqref{est:ui13}, \eqref{est:ui2} hold for $u_{\mathrm{i}}^k$ with $m_k= m_{k-1}-6$, and similarly
\[
\| \p_x^{m_x} \p_y^{m_y} \p_z^{m_z} \rho^k_{\mathrm{i}}\|_{L^2_{x,y}} \lesssim \beta^k  (e^{-\gamma z} + \mu^{m_z} e^{-\mu z}).
\]

\medskip

\noindent\textit{\textbf{Iterative construction of the Munk boundary layer part.}} The construction of Munk boundary layer terms is identical: each term $p^k_{\mathrm M}$ satisfies an equation of the form
\[
\p_t \Delta p^k_{\mathrm M} + \beta \p_1 p^k_{\mathrm M} - \Delta_\nu \Delta_h p^k_{\mathrm M}= F^k_{\mathrm M},
\]
where the term $F^k_{\mathrm M}$ is determined, as before, in terms of $u^j_{\mathrm M}$ for $j\leq k-1$. 
We then obtain $(u^k_{\mathrm{M}}, \rho^k_{\mathrm{M}})$ thanks to the expression \eqref{eq:uk-recurrence}.

There are two main differences with the construction of the interior term:
\begin{itemize}
\item First, following \cref{prop:QG3D-visc}, we will take
\[
p^k_{\mathrm{M}} = G \ast F^k_{\mathrm M} + \mathcal F^{-1}\left(\sum_{j=1,2} c^k_j e^{-\mu_j^+ z}\right),
\]
where the coefficients $c^k_j$ will be determined later. We merely anticipate that $| c^k_j| \lesssim \beta^k$.     

\item Second, we need to be careful with the structure of the velocity when performing the estimate of $F^k_{\mathrm{M}}$ and $p^k_{\mathrm M}$. Note that the induction assumptions \eqref{est:uM13}, \eqref{est:uM2} are compatible with the intensification of the Northward velocity.
Let us estimate for instance the first term in the right-hand side of $F^k_{\mathrm M}$. Using the assumptions on $u^{k-1}_{\mathrm M}$, we have
\[
\| (\p_t - \Delta_\nu)^2  \divh u_{\mathrm M, h}^{k-1} (z)\|_{H^s_{x,y}}\lesssim (| \om|^2 +  \nueff^2  \mu^4) \mu \beta^{k-1} e^{-\mu z}.
\]
Since $|\om| \ll \beta$ and $ \mu \lesssim(\beta/\nueff)^{1/3} $, $| \om|^2 +  \nueff^2  \mu^4\ll \beta^2$. Estimating the other terms in the same fashion, we infer that
\[
\| F^k_{\mathrm{M}}(z)\|_{H^{s-5}} \lesssim \beta^{k+1} \mu e^{-\mu  z},
\]
and therefore, using \cref{prop:QG3D-visc} and \cref{lem:exp-Green} (and recalling that $\mu \leq \Re(\mu_j^+)/2$ by definition),
\[
\| p^k_{\mathrm M}\|_{H^{s-5}} \lesssim \beta^k  e^{-\mu z}.
\]
\end{itemize}
From there, we easily obtain the estimates on each component of $u^k$, recalling that $\mu=(\beta/\nueff)^{1/3}$ and that $| \om| \lesssim \beta^{2/3} \nueff^{1/3}.$ We also use and propagate the induction assumption 
\[
\| \rho^k_{\mathrm{M}}\|_{H^s}  \lesssim \beta^k \mu e^{-\mu z}. 
\]
For further purposes, we denote by $\bar u^k_{\mathrm{M}}$ the velocity
\[
\bar u^k_{\mathrm{M}} = \begin{pmatrix}
\nabla_h^\bot (G\ast F^k_\mathrm{M}) + (\p_t - \Delta_\nu) (u_{\mathrm{M}, h}^{k-1})^\bot - \beta y u_{\mathrm{M}, h}^{k-1}\\
\p_t \rho^{k-1}_\mathrm{M}
\end{pmatrix},
\]
so that
\be\label{decomp-uMk}
\mathcal F( u^k_{\mathrm{M}}) = \mathcal F( \bar u^k_{\mathrm{M}} ) + \sum_{j\in \{1,2\}} c_j^k e^{-\mu_j^+ z}\begin{pmatrix}
-i \xi_ y\\ c i \xi_x+ s \mu_j^+\\ 0
\end{pmatrix},
\ee
for some coefficients $c_1^k, c_2^k$ which remain to be determined.
\medskip

\noindent\textit{\textbf{Iterative construction of the Ekman boundary layer part.}}

We now address the construction of the Ekman boundary layer part.
Note that the Ekman boundary layer at main order is an exact solution of \eqref{NS-rho}, up to the $\beta$ term. 
Therefore, we will construct the sequence $u_{\mathrm{E}}^k$ so that at every step,
\begin{align*}
\p_t u_{\mathrm{E},h}^k + \frac{1}{\eps} (u_{\mathrm{E},h}^k)^\bot + \frac{1}{\eps}\nabla_h p_{\mathrm{E}}^k - \Delta_\nu  u_{\mathrm{E},h}^k &= -\frac{\beta y} {\varepsilon} (u_{\mathrm{E},h}^{k-1})^\bot,\\
\p_t u_{\mathrm{E},3}^k + \frac{1}{\varepsilon \delta^2}\p_3 p_{\mathrm{E}}^k - \Delta_\nu u_{\mathrm{E},3}^k + \frac{1}{\varepsilon \delta^2} \rho_{\mathrm{E}}^k&=0,\\
\p_t \rho_{\mathrm{E}}^k - \frac{1}{\varepsilon} u_{\mathrm{E},3}^k &=0,\\
\Div  u_{\mathrm{E}}^k &=0.
\end{align*}
Following the computations of \cref{sec:Ekman}, we find that the tangential Fourier transform of $u_{\mathrm{E},x}^k$ satisfies an ODE of the form
\[
P_{\varepsilon,\om, \nu}(\xi_x,\xi_y, \p_z) \widehat{u_{\mathrm{E},x}^k} = S^{k-1},
\]
where the right-hand side $S^{k-1}$ involves the $\p_{\xi_y}$-derivative of the function $u_{\mathrm{E},h}^{k-1}$. The expression of the polynomial $P_{\varepsilon,\om, \nu}$ is complicated, but fortunately we will not need to compute it exactly. It is sufficient to note that $\lambda_1$ is a root of $P_{\varepsilon,\om, \nu}$. This prompts us to look for each term in the expansion as a polynomial multiplied by an exponential. 
More precisely, since $u_{\mathrm{E}}^0= u_{\mathrm{E}}^1=0$, we have $S^1=0$, and therefore $\widehat{u_{\mathrm{E}}^2}= c_{\mathrm{E}}^2 e^{-\lambda_1 z}\widetilde{\mathsf{U}}^2$ for some coefficient $c_{\mathrm{E}}^2$ which will be determined later.
For the higher order terms, 
 following the computations from \cref{sec:Ekman}, we will take the following Ansatz:
\be\label{Ansatz-Ekman-k}\ba
\widehat{u_{\mathrm{E},x}^k}= \beta^k Q^k_x (Z) e^{-Z},\qquad
\widehat{u_{\mathrm{E},y}^k}= \beta^k \om \eps Q^k_y (Z) e^{-Z},\qquad
\widehat{u_{\mathrm{E},z}^k}= \beta^k \lambda_1^{-1} Q^k_z (Z) e^{-Z},\\
\widehat{\rho_{\mathrm{E}}^k}= \beta^k \frac{1}{i\om \eps} Q^k_\rho (Z) e^{-Z},\qquad
\widehat{p_{\mathrm{E}}^k}= \beta^k \frac{1}{i\eps \om \lambda_1} Q^k_p (Z) e^{-Z},
\end{aligned}\ee
where $Z=\lambda_1 z$, and $Q^k_a$ are complex polynomials in $Z$ with bounded coefficients. Plugging this Ansatz into the equation, we find the following relations between the polynomials:
\begin{align*}
-\p_Z (Q_z^k(Z) e^{-Z}) e^Z= i\xi_x Q_x^k + \om \eps Q_y^k \quad \text{(divergence-free condition),}\\
Q_\rho^k = s Q_x^k + c \lambda_1^{-1} Q_z^k \quad \text{(conservation of mass),}\\
i\om \eps^2 \delta^2 \mathfrak r(\p_Z)\left( (s Q_x^k + c \lambda_1^{-1} Q_z^k)e^{-Z}\right) + Q^k_\rho (Z) e^{-Z} + (c\p_z + i s\xi_x \lambda_1^{-1}) (Q^k_p(Z) e^{-Z})=0 \\\hspace{7cm}{ \text{(vertical momentum balance),}}
    \end{align*}
where $\mathfrak r(\p_Z)= i\om + \nu_h (c^2 \xi_x^2 + \xi_y^2) + \nu_3 s^2 \xi_x^2 - \nueff \lambda_1^2 \p_Z^2 + 2 i cs\xi_x \lambda_1(\nu_h - \nu_3) \p_Z$. Note that the third identity implies in particular that the hydrostatic equilibtrium is satisfied at main order: $Q_\rho^k\simeq- c e^Z \p_Z( Q_p^k (Z) e^{-Z})$. Looking at the second component in the vertical momentum balance, we get at main order
\[
\frac{is^2}{c^2}\p_Z^2 (Q^k_y e^{-Z})e^Z + c Q^k_x  \simeq Q^{k-1}_x.
\]
This allows us to determine $Q_k^y$ in terms of $Q^k_x$ and $Q^{k-1}_x$. Plugging this expression into the first component of the momentum balance, we find an ODE on $Q^k_x$. From there, we deduce all the other polynomials.

Let us stress that the coefficients of order zero in each $Q^k_a$ remain undetermined at this stage, as they belong to the kernel of the differential operator.
They correspond to a multiple of the Ekman layer at main order. Hence, we will write $Q^k_a$ as $\overline{Q^k_a} + c^k_{\mathrm{E}} \mathsf{U}^1_a$, where $\mathsf{U}^1_a$ is the coordinate along the $a$-component of the generalized eigenvector $\mathsf U^1$ associated with $\lambda_1$ (see \cref{sec:Ekman}), 
and the polynomial $\overline{Q^k_a} $ vanishes at $Z=0$.
We denote by $\bar u_{\mathrm E}^k$ the velocity obtained when $Q^k_a$ is replaced by $\overline{Q^k_a}$ in \eqref{Ansatz-Ekman-k}, so that
\be\label{decomp-uEk}
\mathcal F(u_{\mathrm{E}}^k)= \mathcal F(\bar u_{\mathrm{E}}^k) + c^k_{\mathrm{E}} e^{-\lambda_1 z} \widetilde{\mathsf{U}}^1,
\ee
where $\widetilde{\mathsf{U}}$ are the first three components of ${\mathsf{U}}$.

\medskip

\noindent\textit{\textbf{Determination of the boundary layer coefficients $c^k_1, c^k_2, c^k_{\mathrm E}$.}}

We then identify the coefficients $c^k_j$ and $c^k_{\mathrm{E}}$ coming respectively from the Munk boundary layer part and from the Ekman layer at order $k$. To that end, we require that
\[
u^k_{\mathrm i} + u^k_{\mathrm M}  + u^k_{\mathrm E}=0 \qquad\text{on } \p\Om.
\]
Recalling \eqref{decomp-uMk} and \eqref{decomp-uMk}, this implies
\[
 \sum_{j\in \{1,2\}} c_j^k \begin{pmatrix}
-i \xi_ y\\ c i \xi_x+ s \mu_j^+\\ 0
\end{pmatrix} + c^k_{\mathrm{E}} \begin{pmatrix}
    \mathsf{U}^1_1\\  \mathsf{U}^1_2\\   \mathsf{U}^1_3
\end{pmatrix}
= - \mathcal F\left( u^k_{\mathrm i} + \bar u^k_{\mathrm M}  + \bar u^k_{\mathrm E}\right)_{| \p\Om},
\]
where we recall that the right-hand side is completely determined by lower order terms, and $\mathsf U_1^1 \simeq \cos \alpha$, $\mathsf U_2^1 \simeq i \cos^3 \alpha \om \eps/\sin^2 \alpha$, $\mathsf U_3^1 \simeq \sin \alpha$. 

Therefore we first choose the coefficient $c_k^\mathrm E$ so that
\[
 c^k_{\mathrm{E}}= - \frac{1}{\mathsf U_3^1}  \mathcal F\left( u^k_{\mathrm i} + \bar u^k_{\mathrm M}  + \bar u^k_{\mathrm E}\right)\cdot e_3 \vert_{ \p\Om}.
\]
It follows from the previous estimates that $ |c^k_{\mathrm{E}}|\lesssim \beta^k$. The term $u^k_{\mathrm E}$ is now fully determined and satisfies the estimates announced in the statement of the Lemma.

We then compute the coefficients $c_j^k$ for $j\in \{1,2\}$ by inverting the matrix in the first term of the left-hand side. More precisely,
\[
\begin{pmatrix}
c_1^k\\
    c_2^k
\end{pmatrix}
=
-\frac{1}{i\xi_y s (\mu_2 - \mu_1)} \begin{pmatrix}
c i \xi_x+ s \mu_2 & i \xi_y \\
- c i \xi_x- s \mu_1 & - i \xi_y 
    \end{pmatrix}
\mathcal F\left( u^k_{\mathrm i} + \bar u^k_{\mathrm M}  +  u^k_{\mathrm E}\right)_{h| \p\Om}
\]
It follows from the estimates on $u^k_{\mathrm i} $, $ \bar u^k_{\mathrm M} $ and $ u^k_{\mathrm E}$, and from the assumptions on the support in Fourier of $u^k_a$ that $c_j^k= O(\beta^k).$ Hence the $k$-th Munk corrector is now fully determined, which completes the construction of $u^k$.

\medskip

\noindent\textit{\textbf{Evaluation of the remainder and conclusion.}} 

We now evaluate the remainder associated with each family.

\begin{itemize}
    \item The remainder associated with the interior part is
    \[
\begin{pmatrix}
\eps^K (\p_t - \Delta_\nu) u_{\mathrm i, h}^{K} + \eps^K \beta y (u_{\mathrm i, h}^K)^{\bot}\\
\eps^{K-1} (\p_t - \Delta_\nu) u_{\mathrm i, 3}^{K-1} + \eps^K (\p_t - \Delta_\nu) u_{\mathrm i, 3}^{K}\\
\eps^K \p_t \rho^K_{\mathrm i}.
\end{pmatrix}    \]
According to the previous estimates, the size of this term in $L^\infty_t(L^2(\Om, (1+z^2) \dd x \dd y \dd z))$ is
\[
\begin{pmatrix}
O(\beta (\eps \beta)^K)\\
O (\om (\eps \beta)^{K-1})\\O (\om (\eps \beta)^{K})
\end{pmatrix}.
\]
Recalling the assumptions on the parameters $\beta$ and $\om$ and choosing $K$ sufficiently large so that $-a + K(1-a) \geq N$ and $-b + (K-1)(1-a) \geq N$, we find that the right hand side is $O(\eps^N)$.

\item In a similar way, we now compute the size of the remainder associated with the Munk boundary layer part. The expression of the remainder is identical to the one of the interior term, replacing the subscript $\mathrm{i}$ by $\mathrm{M}$.
Recalling the estimates on $u_\mathrm{M}^k$ and using the estimate 
\[
\| e^{-\mu z} \|_{L^2_z}\lesssim \mu^{-1/2},
\]
we find that the size of the remainder associated with $u_{\mathrm M}$ is
\[
\begin{pmatrix}
O(\beta \mu^{1/2} (\eps \beta)^K)\\
    O (\om \mu^{-1/2}(\eps \beta)^{K-1})\\
    O (\om \mu^{1/2} (\eps \beta)^{K})
\end{pmatrix}
\]
Once again, choosing $K$ large enough, each error term is $O(\eps^N)$.

\item Eventually, the remainder term associated with the Ekman boundary layer part is
\[
\begin{pmatrix}
 \eps^K \beta y (u_{\mathrm E, h}^K)^{\bot}\\ 0\\0
    \end{pmatrix}
    = \begin{pmatrix}
O( \beta (\eps \beta)^K (\Re \lambda_1)^{-1/2})\\0\\0
    \end{pmatrix} = O(\eps^N)
\]
for $K$ large enough. 

Let us finally comment on the regularity of the solution and of the data. A finite number of derivatives are ``consumed'' with each step of the construction, namely, if $u^k\in H^s$, then $u^{k+1} \in H^{s-6}$. As a consequence, if we wish to ensure that $u^k\in H^m$ for some $m\geq 0$ for all $k\in \{0, \cdots K\}$, then we must have $u^0 \in H^{m+ 6 K} $ and thus $f\in H^{m+6K}$. Hence, choosing $f\in H^s$ for some $s$ large enough depending on $N$, we obtain the desired result.
\end{itemize}
\end{proof}

\begin{remark}
    \label{rem:delta-m}
If the aspect ratio $\delta$ is such that $\delta= \eps^{M/2}$ for some $M\in \N$, with $M$ possibly different from 2, the above strategy remains valid with minor adjustements. The main difference lies in the fact that \eqref{eq:u3k} needs to be changed into
\[
(\p_t - \Delta_\nu) u_3^k + \p_3 p^{k+M+1}= - \rho^{k+M+1}.
\]
Therefore the term involving $u_3^j$ for $j<k$ in $F^k$ needs to be modified, and becomes $\p_t \p_3(\p_t-\Delta_\nu)u_3^{k-M-1}$. 

The spirit of the proof and the estimates remain otherwise unchanged.
\end{remark}

\subsection{Truncation of the source term}

Our construction of an approximate solution relies on the assumption that the source term $f$ is compactly supported in Fourier space. It is easy to get rid of this assumption, provided the soure term is sufficiently smooth:
\begin{lemma}
	Assume that \eqref{hyp:parameters} is satisfied, and let $R= \eps^{-\kappa}$ with $\kappa$ such that $0<\kappa<(a+d)/3$ (i.e. $1\ll R \ll (\beta/\nu_h)^{1/3}$).

	 Let $N\geq 0$ be arbitrary. 
There exists $s\geq 0$ such that if $S\in H^s(\R^2)$, then
\[
\left\| \cF^{-1} (\mathbf 1_{|\xi| >R} \hS(\xi)) \right\|_{L^2(\R^2)} \leq \eps^N \| S\|_{H^s}.
\]

\label{lem:truncation-source}
\end{lemma}
\begin{proof}
The Cauchy-Schwarz inequality ensures that
\[
\left\| \cF^{-1} (\mathbf 1_{|\xi| >R} \hS(\xi)) \right\|_{L^2(\R^2)} \leq \| S \|_{H^s} \left(\int_{|\xi|>R} |\xi|^{-2s} \dd \xi\right)^{1/2} \lesssim  \| S \|_{H^s} R^{1-s}.
\]
Picking $s$ such that $\kappa (s-1)\geq N$, we obtain the desired result.
\end{proof}

\subsection{Periodic stability}
We are now ready to prove Theorem \ref{thm:stab-periodic}. We first truncate the source term $f$ for large frequencies as in \cref{lem:truncation-source}, so that the remainder is $O(\eps^N)$ in $L^2((1+z^2)\dd x\dd y\dd z)  $.

Let $(u,\rho)$ be an exact, periodic solution of \eqref{NS-rho}, and let $(\uapp ,\rho^\text{app})$ be the approximate solution constructed in \cref{lem:approx-high-order} with $m=1$ and $N$ te be determined, and with $f$ replaced by $\cF^{-1} (\mathbf 1_{|\xi| \leq R} \widehat{f})$.
We set $v=u-\uapp$, $\varphi =\rho-\rho^{\mathrm{app}}$.
Then $(v, \varphi)$ is a solution of \eqref{NS-rho} with the remainder $g_{\mathrm{rem}} + \cF^{-1} (\mathbf 1_{|\xi| >R} \widehat{f}(\xi))= g'_{\mathrm{rem}}$, and $g'_{\mathrm{rem}}= O(\eps^N) $ in $L^2((1+z^2)\dd x\dd y\dd z)  $.

We now take the scalar product of \eqref{NS-rho} (written for $(v, \varphi)$) with $(v_h, \delta^2 v_3, \varphi)$ and integrate over $\Om$, using the no-slip boundary condition and the incompressibility of the fluid. We obtain, for every $t\in \R$,
\be\label{est:energy}\ba
&\frac{1}{2} \frac{d}{dt}\int_\Om (| v_h|^2 + \delta^2 | v_3|^2 + |\varphi|^2) + \int_\Om (\nu_h | \nabla_h v_h|^2 + \nu_h \delta^2 |\nabla_h v_3|^2 + \nu_3 | \p_3 v_h|^3 + \nu_3 \delta^2 | \p_3 v_3|^2)\\
\leq &\int_\Om |g'_{\mathrm{rem}}\cdot (v_h, \delta^2 v_3, \varphi)|.\ea
\ee
We then integrate in time and obtain
\[
\int_0^T \int_\Om (\nu_h | \nabla_h v_h|^2 + \nu_h \delta^2 |\nabla_h v_3|^2 + \nu_3 | \p_3 v_h|^3 + \nu_3 \delta^2 | \p_3 v_3|^2)
\leq \int_0^T\int_\Om |g'_{\mathrm{rem}}\cdot   (v_h, \delta^2 v_3, \varphi)|.
\]
We then make the following observations:
\begin{itemize}
    \item First, according to the Hardy inequality, for $i\in \{1,2,3\}$, since $v\vert_{z=0}=0$,
    \[
\int_0^\infty \frac{1}{z^2}| v_i(\cdot, z)|^2 \dd z \leq C \int_0^\infty (\p_z v_i(\cdot, z))^2\dd z.
        \]
As a consequence,
\[
\int_0^T\int_\Om |g'_{\mathrm{rem}, i}\cdot  v_i|\lesssim \left(\int_0^T \int_\Om | \p_z v_i|^2 \right)^{1/2} \left(\int_0^T \int_\Om z^2 |g'_{\mathrm{rem}, i}(t,x,y,z)|^2 \dd x\,\dd y\,\dd z\,\dd t\, \right)^{1/2}.
\]

\item Furthermore, using the Poincaré inequality together with the fact that the time averages of $\rho$ and $\rho^{\mathrm{app}}$ vanish, we have
\[
\int_0^T | \varphi(t,\cdot)|^2 \dd t \leq C_T \int_0^T | \p_t \varphi(t,\cdot )|^2 \dd t \lesssim\int_0^T(\eps^{-2} | v_3|^2 + |g'_{\mathrm{rem, 4}}|^2)\dd t.
\]

\end{itemize}
We infer that
\[
\delta^2 \inf(\nu_h, \nu_3) \| \nabla v \|_{L^2((0,T)\times \Om}\lesssim \eps^{-1} \left(\int_0^T \int_\Om (1+|z|^2) |  g'_{\mathrm{rem}}|^2\dd x \dd y \dd z \dd t \right)^{1/2}.
\]
Choosing $N$ sufficiently large in \cref{lem:approx-high-order}, we deduce
\[
\| v\|_{H^1_{x,y}, L^2_{t,z}}\leq \| \nabla v\|_{L^2((0,T)\times \Om}\lesssim (\eps \beta)^2.
\]
Observing that $\|\eps^k u^k_{1,3}\|_{H^1_{x,y}, L^2_{t,z}}\lesssim (\eps \beta)^k $, $\|\eps^k u^k_{2}\|_{H^1_{x,y}, L^2_{t,z}} \lesssim (\eps \beta)^k (\beta/\nueff)^{1/6}$,
this completes the proof of Theorem \ref{thm:stab-periodic}.\qed

\subsection{Stability for the Cauchy problem}

The proof of Theorem \ref{thm:stab-Cauchy} goes along the same lines as the ones of \cref{thm:stab-periodic}, and is in fact slightly easier. We follow the computations of the previous paragraph and start from the energy inequality \eqref{est:energy}.
Setting
\[
E_\eps(t):= \| v_h(t)\|_{L^2(\Om)}^2 + \delta^2 \| v_3(t)\|_{L^2(\Om)}^2 + \| \varphi(t)\|_{L^2(\Om)}^2,
\]
we have
\[
\frac{d E_\eps}{dt}\lesssim E_\eps^{1/2} \| g'_{\mathrm{rem}}(t)\|_{L^2(\Om)},
\]
and thus
\[
E_\eps(t)^{1/2} \lesssim E_\eps(0)^{1/2} + \| g'_{\mathrm{rem}}\|_{L^1((0,t), L^2(\Om))}\lesssim (\eps \beta)^2(1+t)
\]
provided $N$ is chosen large enough.\qed

\section*{Acknowledgements}

The authors thank Nina Aguillon, Julie Deshayes, Sima Dogan, Stephen Griffies, Julien Guillod  and Gurvan Madec for nice discussions about this problem.
 This work was supported by the BOURGEONS project, grant  ANR-23-CE40-0014-01 of the French National Research Agency (ANR), and has benefited from a government grant managed by the Agence Nationale de la Recherche under the France 2030 investment plan ANR-23-EXMA-0003 (project Climath).
A.-L.\ D.\ acknowledges the support of the Institut Universitaire de France.

\appendix
\section{Formal derivation of \texorpdfstring{\eqref{NS-rho}}{the linearized system}}
\label{appendix-derivation}

We start from the incompressible, density dependent Navier--Stokes system with rotation. We will write the original physical variables with a tilde $\widetilde{\ }$, and the dimensionless ones without tilde.
Hence the original system is
\be\label{NSC-dimension}
\ba
\widetilde{\rho}(\widetilde{\p_t} \tu + (\tu \cdot \widetilde{\nabla} )\tu) +  \widetilde{\nabla} \widetilde{p} + 2 \widetilde{\rho}\widetilde{\Gamma} \mathbf{e}\wedge \tu - \widetilde{\nu}_h \widetilde{\Delta}_h \tu - \widetilde{\nu}_3 \widetilde{\p_3}^2 \tu &\; =&  \widetilde{ \rho}\mathbf{\widetilde{g}},\\
\widetilde{\Div} \tu  &\; =&0,\\
\widetilde{\p_t} \widetilde{ \rho} + \widetilde{\Div}( \widetilde{ \rho} \tu) &\; =&0,\\
\ea
\ee
where $\tu$ is the velocity of ocean currents, $\widetilde{ \rho} $ the seawater density, $\widetilde{\Gamma}$ the angular speed of Earth rotation and $\mathbf{e}$ the unitary vector directed from the South pole to the North pole, and $\mathbf{\widetilde{g}}$ is the gravitational acceleration.
Note that the diffusion operator stems from a classical turbulent description of small scales in oceanography: the interactions between small vortices is expected to dissipate energy, through a mechanism which is deemed as analogous to collisions between particles in an ideal gas.
Hence the coefficients $\widetilde{\nu}_h$, $\widetilde{\nu}_3$ differ from the molecular viscosity of seawater, and are called ``eddy diffusivities''. 
Since the motion of the fluid is strongly anisotropic (the motion is horizontal at main order, as recalled in \cref{sec:heuristique}), the diffusion tensor is also anisotropic, and $\widetilde{\nu}_h\gg\widetilde{\nu}_3$ \textit{a priori}.

Let us now write the system in dimensionless form. 
We consider a cartesian coordinate frame, centered around a given tempered latitude $\theta_0\in (0,\pi/2)$ measured from the equator.
As mentioned before, the vector $e_1$ is the normalised Eastward vector, $e_2$ the Northward one and $e_3$ the local vertical, so that $\mathbf{e}=\cos \theta e_1 +  \sin \theta e_3$. 
As a consequence, we have $\widetilde{y}= r_* \tan(\theta-\theta_0)$, where $r_*$ is the Earth radius abd $\theta$ is the latitude, and thus $y=\frac{r_*}{L_*} \tan(\theta-\theta_0)\simeq \frac{r_*}{L_*} (\theta-\theta_0)$, where $L_*$ is the typical horizontal length scale. 
We also denote by $D_*$ the typical depth, so that $\delta= D_*/L_*$ is the aspect ratio, and by $T_*$ the typical time scale.
As for the unknowns, we write $\tu_h= U_* u_h$, $\widetilde{p}= P_* p$, $\widetilde{\rho}= R_* \rho$, and $\tu_3= \delta U_* u_3$ in order to preserve the divergence free condition. 
The dimensionless system becomes
\be\label{NSC-dimensionless}
\ba
\rho(\p_t u +\frac{U_* T_*}{L_*} (u \cdot \nabla) u )+ \frac{P_* T_* }{R_*L_* U_*} \begin{pmatrix}
\nabla_h p\\ \delta^{-2} \p_3 p\end{pmatrix}
- \nu_h {\Delta}_h u- {\nu}_3 \p_3^2 u  & \\
\qquad
+ 2 \widetilde{\Gamma} T_* \left(\sin \theta e_3 \wedge u + 
\cos \theta \begin{pmatrix}
0\\ -\delta u_3\\ \delta^{-1} u_2
\end{pmatrix}\right)&\; =  \mathsf{g} \frac{T_*}{\delta U_*} \rho e_3,\\
{\Div} u  &\; =0,\\
\p_t \rho + \frac{U_* T_*}{L_*} \Div(\rho u) &\; =0,\\
\ea
\ee
    where
    \[
    \nu_h= \widetilde{\nu}_h \frac{T_*}{R_*L_*^2},\qquad \nu_3= \widetilde{\nu}_3 \frac{T_*}{R_*D_*^2}.
    \]

Let us now provide some orders of magnitude on the different dimensionless coefficients, and make some further assumptions:
\begin{itemize}
    \item We will work on large horizontal length scales, typically $L_*\sim 5 \cdot 10^2 \; \mathrm{km} $, while the average depth of the ocean is $D_*\sim 5 \mathrm{km}$. Thus $\delta \sim 10^{-2}$.
    \item A typical value of horizontal velocities in oceanic currents is $U_*\sim 10^{-1} \mathrm{m\cdot s^{-1}}$. We will take $R_*\sim 10^3 \mathrm{kg\cdot m^{-3}}$ and $P_*=R_* \mathsf{g} D_*$.

\item We choose $T_*$ to be a fraction of the advective time scale, namely $T_*= 10^{-1} L_*/U_*\sim 6$ days. With this choice, the Rossby number, defined as
\[
\eps:= \frac{1}{2\tilde \Gamma T_* \sin \theta_0}
\]
is such that $\eps \simeq 10^{-2}$,    while $P_* T_*/(R_*L_* U_*) = 10^{-1} \mathsf{g} D_*/U_*^2\sim 5\cdot 10^5\propto \eps^{-3}$.

    \item Linearizing the Coriolis factor around the latitude $\theta_0$, we have
    \[
    \sin \theta\simeq\sin \theta_0 + \cos \theta_0 (\theta-\theta_0)\simeq \sin \theta_0 \left( 1 + \cot \theta_0 \frac{L_*}{r_*} y\right).
    \]
    We therefore set
    \[
    \beta:= \cot \theta_0 \frac{L_*}{r_*\eps},
    \]
and we note that $\beta\sim \eps^{-1/2}$ in the present scaling, which is consistent with \eqref{hyp:parameters} and which corresponds to the parameter $\beta$ used to produce \cref{fig:illustration}.
    \item A range of  values for the eddy diffusivities may be found in \cite[Sections 9.6 and 12.6]{gill2016atmosphere}; one can take for instance
    \[
    \widetilde{\nu}_3 R_*^{-1}= 3\cdot 10^{-3}\mathrm{m^2\cdot s^{-1}},\qquad  \widetilde{\nu}_h R_*^{-1} \in [10^2, 10^4]\mathrm{m^2\cdot s^{-1}},
    \]
    leading to
    \[
    \nu_3\simeq 10^{-4},\qquad \nu_h \in [2\cdot 10^{-4}, 2 \cdot 10^{-2}].
    \]
\item Eventually, we assume that the density is a small variation around a constant value, and is stably stratified. More precisely, we take
\[
\rho= 1 + \eps \bar \rho + \eps^2 \rho',
\]
and accordingly $p= 1-z + \eps \bar p + \eps^2 p'$, with $\p_z \bar p= - \bar \rho$. 

The density equation then becomes
\[
\p_t \rho' + \frac{1}{10 \eps^2}u_3\p_3 \bar \rho + \frac{1}{10} \Div(\rho u')=0. 
\]
    
\end{itemize}
Eventually, system \eqref{NS-rho} follows (after omitting the primes in the pressure and density variations $P', \rho'$) by making the following final assumptions and approximations:
\begin{itemize}
    \item The $\cos \theta$ term is neglected, since its contribution to the vertical momentum balance is of lower order than the stratification and vertical pressure gradient, and its contribution to the horizontal momentum blanace is small (traditional approximation);

    \item The nonlinear terms $u\cdot \nabla u$ and $\Div(\rho' u)$ are discarded, and we make the usual Boussinesq approximation $\rho\p_t u \simeq \p_t u$;

    \item The stratification is assumed to be linear at main order, and such that $\p_3 \bar \rho=-10$.
\end{itemize}
Note that the above derivation leads to $\delta\sim \eps$, $\beta\sim \eps^{-1/2}$, $\nu_3\sim \eps^2$ and $\nu_h \in [2 \eps^2, 2 \eps]$, which is consistent with \eqref{hyp:parameters}.

We conclude this Appendix with a definition of weak solution for system \eqref{NS-rho}.
\begin{definition}[Weak solutions]
Let $T>0$, and let $u\in C([0,T], L^2(\Om))^3\cap L^2((0,T), H^1(\Om))^3$, $\rho\in C([0,T], L^2(\Om))$.
Assume that $\Div u=0$ almost everywhere. We say that $(u,\rho)$ is a weak solution of \eqref{NS-rho} associated with the initial data $(u_{\mathrm{ini}}, \rho_{\mathrm{ini}})$ if, for any $v\in C^1([0,T], L^2(\Om))^3\cap L^2((0,T), H^1(\Om))^3$ such that $\Div v=0$, for any $\varphi\in  C^1([0,T], L^2(\Om))$, for any $t\in (0,T)$,
\begin{eqnarray*}
&&\int_{\Om}\left(u_h(t)\cdot v_h(t) + \delta^2 u_3(t)v_3(t) + \rho(t) \varphi(t)\right) \\
&&+ \eps^{-1} \int_0^t \int_\Om (1+ \eps \beta y) e_3\wedge u \cdot v + \eps^{-1} \int_0^t \int_\Om (\rho v_3 - \varphi u_3)\\
&&+ \int_0^t \int_\Om \left[\nu_h (\nabla_h u_h : \nabla_h v_h + \delta^2 \nabla_h u_3\cdot \nabla_h v_3) + \nu_3 (\p_3 u_h\p_3 v_h + \delta^2 \p_3 u_3 \p_3 v_3)\right] \\
&=&\beta \int_0^t \int_\Om \left( f_h\cdot v_h + \delta^2 f_3\cdot v_3\right) + \int_0^t\int_\Om  \left(u_h\cdot \p_t v_h + \delta^2 u_3\p_t v_3 + \rho \p_t \varphi\right)\\
&&+ \int_{\Om}\left(u_{\mathrm{ini},h}\cdot v_{h}(t=0) + \delta^2 u_{\mathrm{ini},3}v_{3}(t=0)+ \rho_{\mathrm{ini}} \varphi(t=0)\right).
\end{eqnarray*}

\end{definition}

\bibliography{biblio-separation-maths}

\end{document}